\documentclass{amsart}
\usepackage{amsmath,amssymb}

\newcommand{\IZ}{\mathbb Z}
\newcommand{\IN}{\mathbb N}
\newcommand{\IR}{\mathbb R}

\newcommand{\e}{\varepsilon}

\newcommand{\conv}{\mathrm{conv}}

\newcommand{\pr}{\mathrm{pr}}
\newcommand{\2}{\mathbf 2}
\newcommand{\supp}{\mathrm{supp}}
\newcommand{\w}{\omega}
\newcommand{\M}{\mathcal M}
\newcommand{\cov}{\mathrm{cov}}

\newcommand{\LLC}{\mathsf{ILC}}
\newcommand{\U}{\mathcal U}
\newcommand{\binfty}{{\infty\hskip-0.98em\infty}}
\newcommand{\Xiup}{\mathrm{\Xi}}
\newcommand{\Sigmaup}{\mathrm{\Sigma}}

\newtheorem{theorem}{Theorem}

\newtheorem{claimm}{Claim}
\newtheorem{lemma}{Lemma}

\newtheorem{proposition}{Proposition}

\newtheorem{problem}{Problem}
\theoremstyle{definition}
\newtheorem{definition}{Definition}

\title{Centerpole sets for colorings of Abelian groups}
\author{Taras Banakh \and  Ostap Chervak} 
\address{Ivan Franko National University of Lviv, Ukraine, and\newline
Uniwersytet Humanistyczno-Przyrodniczy Jana Kochanowskiego, Kielce, Poland}
\email{t.o.banakh@gmail.com, \; oschervak@gmail.com}

\begin{document}

\begin{abstract}
A subset $C\subset G$ of a topological group $G$ is called {\em $k$-centerpole\/} if for each $k$-coloring of $G$ there is an unbounded monochromatic subset $G$, which is symmetric with respect to a point $c\in C$ in the sense that $S=cS^{-1}c$. By $c_k(G)$ we denote the smallest cardinality  $c_k(G)$ of a $k$-centerpole subset in $G$. We prove that $c_k(G)=c_k(\IZ^{m})$ if $G$ is a discrete abelian group of free rank $m\ge k$. Also we prove that 
$c_1(\IZ^{n+1})=1$, $c_2(\IZ^{n+2})=3$, $c_3(\IZ^{n+3})=6$, $8\le c_4(\IZ^{n+4})\le c_4(\IZ^4)=12$ for all $n\in\w$, and \ 
$\textstyle{\frac12(k^2+3k-4)\le c_k(\IZ^n)\le 2^k-1-\max\limits_{s\le k-2}\binom{k-1}{s-1}}$ \ for all \ $n\ge k\ge 4$.
\end{abstract}
\maketitle

\section{Introduction}

In \cite{BP} T.Banakh and I.Protasov proved that for any $k$-coloring $\chi:\IZ^k\to k=\{0,\dots,k{-}1\}$ of the abelian group $\IZ^k$ there is an infinite monochromatic subset $S\subset\IZ^k$ such that $S-c=c-S$ for some point $c\in\{0,1\}^k$. The equality $S-c=c-S$ means that the set $S$ is symmetric with respect to the point $c$. On the other hand, a suitable partition of $\IR^k$ into $k+1$ convex cones determines a Borel $(k+1)$-coloring of $\IR^k$  without unbounded monochromatic symmetric subsets. 
These two results motivate the following definition, cf. \cite{Ba}, \cite{BDR}.
 
\begin{definition}\label{d1} A subset $C$ of a topological group $G$ is called 
{\em $k$-centerpole\footnote{So, a centerpole set can be thought as a set of poles of central symmetries that detect unbounded monochromatic symmetric subsets.} for (Borel) colorings of $G$} if for any (Borel) $k$-coloring $\chi:G\to k$ of $G$ there is an unbounded monochromatic subset $S\subset G$, symmetric with respect to some point $c\in C$ in the sense that $Sc^{-1}=cS^{-1}$.

The smallest cardinality $|C|$ of such a $k$-centerpole set $C\subset G$ is denoted by $c_k(G)$ (resp. $c_k^B(G)$). If no $k$-centerpole set $C\subset G$ exists then we write $c_k(G)=\binfty$ (resp. $c_k^B(G)=\binfty$) and assume that $\binfty$ is greater than any cardinal that appears in our considerations. 
\end{definition}

Now we explain some terminology that appears in this definition. 
A subset $B$ of a topological group $G$ is called {\em totally bounded} if $B$ can be covered by finitely many left shifts of any neighborhood $U$ of the neutral element of $X$. In the opposite case $B$ is called {\em unbounded}. 

A cardinal number $k$ is identified with the set $\{\alpha:|\alpha|<\kappa\}$    of ordinals of smaller cardinality and endowed with the discrete topology. 

By a ({\em Borel\/}) {\em $k$-coloring} of a topological space $X$ we mean  a (Borel) function $\chi:X\to k$. A function $\chi:X\to k$ is {\em Borel\/} if  for every color $i\in k$ the set $\chi^{-1}(i)$ of points of color $i$ in $X$ is Borel. 

The definition of the numbers $c_k(G)$ and $c_k^B(G)$ implies that 
$$c_k^B(G)\le c_k(G)$$ for any topological group $G$ and any cardinal number $k$.
If the topological group $G$ is discrete, then each coloring of $G$ is Borel, so $c_k^B(G)=c_k(G)$ for all $k$. 
In general, the cardinal numbers $c_k(G)$ and $c_k^B(G)$ are different. For example, $c_\w^B(\IR^\w)=\w$ while $c_\w(\IR^\w)=\binfty$, see Theorem~\ref{t:exact}.

It follows from the definition that $c_k(G)$ and $c_k^B(G)$ considered as functions of $k$ and $G$ are non-decreasing with respect to $k$ and non-increasing with respect to $G$. More precisely, for a number $k\in\IN$, a topological group $G$ and its subgroup $H$ we have the inequalities
$$c_k(H)\ge c_k(G),\;\;c_k(G)\le c_{k+1}(G)\quad\mbox{and}\quad 
c_k^B(H)\ge c_k^B(G),\;\;c_k^B(G)\le c_{k+1}^B(G)
.$$
In the sequel we shall use these monotonicity properties of $c_k(G)$ and $c_k^B(G)$ without any special reference.

In this paper we investigate the problem of calculating the numbers $c_k(G)$ and $c_k^B(G)$ for an abelian topological group $G$  and show that in many cases this problem reduces to calculating the numbers $c_k(\IR^n\times\IZ^{m-n})$ and $c_k^B(\IR^n\times\IZ^{m-n})$ where $n=r_\IR(G)$ is the $\IR$-rank and $m=r_\IZ(G)$ is the $\IZ$-rank of the group $G$. 

For topological groups $G$ and $H$ the {\em $H$-rank} $r_H(G)$ of $G$ is defined as
$$r_H(G)=\sup\{k\in\w:H^k\hookrightarrow G\}$$where $H^k\hookrightarrow G$ means that $H^k$ is topologically isomorphic to a subgroup of the topological group $G$.
It is clear that $r_\IR(G)\le r_\IZ(G)$ for each topological group $G$.
\smallskip

It is interesting to remark that the $\IZ$-rank appears in the formula for calculating the value of the function $$\nu(G)=\min\{\kappa:c_k(G)=\binfty\}$$introduced and studied in \cite{Pro} and \cite{BP}. By \cite{BP}, for any discrete abelian group $G$
$$\nu(G)=\begin{cases}
\max\{|G[2]|,\log|G|\}&\mbox{if $G$ is uncountable or $G[2]$ is infinite},\\
r_\IZ(G)+1&\mbox{if $G$ is finitely generated},\\
r_\IZ(G)+2&\mbox{otherwise}.\\
\end{cases}
$$
Here $G[2]=\{x\in G:2x=0\}$ is the {\em Boolean subgroup} of $G$ and $\log|G|=\min\{\kappa:|G|\le 2^\kappa\}$.
\smallskip

A topological group $G$ is called {\em inductively locally compact} (briefly, an {\em $\LLC$-group}) if each finitely generated subgroup $H\subset G$ has locally compact closure in $G$.
The class of $\LLC$-groups includes all locally compact groups and all closed subgroups of topological vector spaces.

Our aim is to calculate the numbers $c_k(G)$ and $c_k^B(G)$ for an abelian $\LLC$-group.  First, let us exclude two cases in which these numbers can be found in a trivial way.

One of them happens if the number of colors is 1. In this case
$$c_1^B(G)=c_1(G)=\begin{cases}1&\mbox{if $G$ is not totally bounded},\\
\binfty&\mbox{if $G$ is totally bounded}.
\end{cases}
$$
The other trivial case happens if the  Boolean subgroup
$G[2]=\{x\in G:2x=0\}\subset G$ is unbounded in $G$. In this case, for each finite coloring $\chi:G\to k$ there is a color $i\in k$ such that the set $S=G[2]\cap\chi^{-1}(i)$ is unbounded. Since $S=-S$, we conclude that $S$ is a unbounded monochromatic symmetric subset with respect to $0$, which means that the singleton $\{0\}$ is $k$-centerpole in $G$ and thus 
$$\mbox{$c_k(G)=c_k^B(G)=1$ for all $k\in\IN$.}$$

It remains to calculate the values of the cardinal numbers 
$c_k(G)$ and $c_k^B(G)$ for $k\ge 2$ and an abelian topological group $G$ with totally bounded Boolean subgroup $G[2]$. 

The following theorem reduces this problem of calculation of $c_k(G)$ to the case of the group $\IR^n\oplus\IZ^{m-n}$ where $n=r_\IR(G)$ and $m=r_\IZ(G)$.

\begin{theorem}\label{t:ckG} Let $k\in\IN$ and $G$ be an  abelian  $\LLC$-group $G$ with totally bounded Boolean subgroup $G[2]$ and ranks $n=r_\IR(G)$ and $m=r_\IZ(G)$. Then
\begin{enumerate}
\item[\textup{(1)}] $c_k(G)=c_k(\IR^n\times\IZ^{m-n})$  if $k\le m$, and
\item[\textup{(2)}] $c_k(G)\ge\w$  if $k>m$.
\end{enumerate}
If the topologcal group $G$ is metrizable, then
\begin{enumerate}
\item[\textup{(3)}] $c_k^B(G)=c_k^B(\IR^n\times\IZ^{m-n})$ if $k\le m$, and
\smallskip
\item[\textup{(4)}] $c_k^B(G)\ge \w$ if $k>m$.
\end{enumerate}
\end{theorem}

Here we assume that $\w-\w=0$ and $\w-n=\w$ for each $n\in\w$.

Theorem~\ref{t:ckG} will be proved in Section~\ref{s:Pf:lc}. It reduces the problem of calculation of the numbers $c_k(G)$ and $c_k^B(G)$ to calculating these numbers for the groups $\IR^n\times\IZ^{m-n}$ where $n\le m$. The latter problem turned out to be highly non-trivial. In the following theorem we collect all the available information on the precise values of the numbers $c_k(\IR^n\times\IZ^m)$ and $c_k^B(\IR^n\times\IZ^m)$.

\begin{theorem}\label{t:exact} Let $k,n,m$ be cardinal numbers. 
\begin{enumerate}
\item[\textup{(1)}] If $n+m\ge 1$, then $c_1^B(\IR^n\times\IZ^{m})=c_1(\IR^n\times\IZ^m)=1$.
\item[\textup{(2)}] If $n+m\ge 2$, then $c_2^B(\IR^n\times\IZ^{m})=c_2(\IR^n\times\IZ^m)=3$.
\item[\textup{(3)}] If $n+m\ge 3$, then $c_3^B(\IR^n\times\IZ^{m})=c_3(\IR^n\times\IZ^m)=6$.
\item[\textup{(4)}] If $n+m=4$, then $c_4^B(\IR^n\times\IZ^{m})=c_4(\IR^n\times\IZ^m)=12$.
\smallskip
\item[\textup{(5)}] If $k\ge n+m+1<\w$, then $c_k^B(\IR^n\times\IZ^m)=\binfty$.
\item[\textup{(6)}] If $k\ge n+m+1$, then $c_k(\IR^n\times\IZ^m)=\binfty$.
\item[\textup{(7)}] If $n+m\ge \w$ and $\w\le k<\cov(\M)$, then $c_k^B(\IR^n\times\IZ^m)=\w$.
\end{enumerate}
\end{theorem}

In the last item  by $\cov(\M)$ we denote the smallest cardinality of the cover of the real line by meager subsets. It is known that $\aleph_1\le\cov(\M)\le\mathfrak c$ and the equality $\cov(\M)=\mathfrak c$ is equivalent to the Martin Axiom for countable posets, see \cite[19.9]{JW}. 

The equality $c_4(\IZ^4)=12$ from the statement (4) of Theorem~\ref{t:exact} answers the problem of the calculation of $c_4(\IZ^4)$ posed in \cite{Ba} and then repeated in \cite[Problem 2.4]{BPOP}, \cite[Problem 12]{BVV}, and \cite[Question 4.5]{LvivSem}. 

Theorem~\ref{t:exact} presents all cases in which the exact values of the cardinals $c_k^B(\IR^n\times\IZ^{m-n})$ and $c_k(\IR^n\times\IZ^{m-n})$ are known. In the remaining cases we have some upper and lower bounds for these numbers.
 Because of the inequalities
$$c_k^B(\IR^m)\le c_k^B(\IR^n\times \IZ^{m-n})\le c_k(\IR^n\times\IZ^{m-n})\le c_k(\IZ^m),$$
we see that the upper bounds for the numbers $c_k^B(\IR^n\times\IZ^{m-n})$ and $c_k(\IR^n\times\IZ^{m-n})$ would follow from the upper bounds for the numbers $c_k(\IZ^m)$ while lower bounds from lower bounds on $c_k^B(\IR^m)$.

\begin{theorem}\label{t:bounds} For any numbers $k\in\IN$ and $n,m\in\IN\cup\{\w\}$, we get: 
\begin{enumerate}
\item[\textup{(1)}] $c_k(\IZ^m)\le 2^k-1-\max\limits_{s\le k-2}\binom{k-1}{s-1}$ if $k\le m$,
\item[\textup{(2)}] $c_k^B(\IR^k)\ge \frac12(k^2+3k-4)$ if $k\ge 4$,
\item[\textup{(3)}] $c_k^B(\IR^m)\ge k+4$ if $m\ge k\ge 4$,
\smallskip

\item[\textup{(4)}] $c_k^B(\IR^n)<c_{k+1}^B(\IR^{n+1})$ and $c_k(\IR^n)<c_{k+1}(\IR^{n+1})$ if $k\le n$;
\item[\textup{(5)}] $c_k^B(\IR^n\times \IZ^m)<c_{k+1}^B(\IR^n\times\IZ^{m+1})$ and  $c_k(\IR^n\times \IZ^m)<c_{k+1}(\IR^n\times\IZ^{m+1})$ if $k\le n+m$.
\end{enumerate}
\end{theorem}

The binomial coefficient $\binom{k}{i}$ in statement (1) equals $\frac{k!}{i!\,(k-i)!}$ if $i\in\{0,\dots,k\}$ and zero otherwise. The upper bound from this statement improves the previously known upper bound $c_k(\IZ^n)\le 2^k-1$ proved in \cite{Ba}. For $k=m\le 4$ it yields the upper bounds which coincide with the values of $c_k(\IZ^m)$ given in Theorem~\ref{t:exact}.

The lower bound $c_n^B(\IR^n)\ge \frac12(n^2+3n-4)$ from the item (2) improves the previously known lower bound $c_n^B(\IR^n)\ge\frac12(n^2+n)$, proved in \cite{Ba}.
For $n=4$ it gives the lower bound $12\le c_4^B(\IR^4)$, which coincides with the value of $c_4^B(\IR^4)=c_4(\IZ^4)$.

The statement (5) implies that the sequence $(c_k(\IZ^k))_{k=1}^\infty$ is strictly increasing, which answers Question 2 posed in \cite{Ba}.
Theorem~\ref{t:bounds} will be proved in Section~\ref{s:t:bounds} after some preparatory work done in Section~\ref{prepare}.

For every $k\in\IN$ the sequence $(c_k(\IZ^n))_{n=k}^\infty$ is non-increasing and thus it stabilizes starting from some $n$. The value of this number $n$ is upper bounded by the cardinal number $rc_k^B(\IZ^n)$ defined as follows.

For a topological group $G$ and a number $k\in\IN$ let $rc_k^B(G)$ be the minimal possible $\IZ$-rank $r_\IZ(\langle C\rangle)$ of a subgroup $\langle C\rangle$ of $G$ generated a $k$-centerpole subset $C\subset G$ of cardinality $|C|=c^B_k(G)$. If such a set $C$ does not exist (which happens if $c_k^B(G)=\binfty$), then we put $rc_k^B(G)=\binfty$. 

\begin{theorem}[Stabilization]\label{t:stab} Let $k\ge 2$ be an integer and $G$ be an abelian $\LLC$-group with totally bounded Boolean subgroup $G[2]$ and $\IR$-rank $n=r_\IR(G)$. Then
\begin{enumerate}
\item[\textup{(1)}] $c_k(G)=c_k^B(\IZ^\w)$ if $r_\IZ(G)\ge rc_k^B(\IZ^\w)$;
\item[\textup{(2)}] $c_k^B(G)=c_k^B(\IR^n\times\IZ^\w)$ if $G$ is metrizable and $r_\IZ(G)\ge rc_k^B(\IR^n\times\IZ^\w)$;
\item[\textup{(3)}] $c_k^B(G)=c_k^B(\IR^\w)$ if $G$ is metrizable and $r_\IR(\IR)\ge rc^B_k(\IR^\w)$.
\end{enumerate}
\end{theorem}

In light of Theorem~\ref{t:stab} it is important to have lower and upper bounds for the numbers $rc_k(G)$. 

\begin{proposition}\label{p:rc} For any metrizable abelian $\LLC$-group $G$ with totally bounded Boolean subgroup $G[2]$, and a natural number $2\le k\le r_\IZ(G)$ we get
\begin{enumerate}
\item[\textup{(1)}] $rc_k^B(G)=k$ if $k\le 3$ and
\item[\textup{(2)}] $k\le rc_k^B(G)\le c_k^B(G)-3$ if $k\ge 3$.
\end{enumerate}
\end{proposition}


Finally, let us present the $(k+1)$-centerpole subset $\Xiup^{k}_s$ of $\IR^{1+k}$ that contains $2^k-1-\binom{k}{s}$ elements and gives the upper bound from Theorem~\ref{t:bounds}(1). This $(k+1)$-centerpole set $\Xiup_{k}$ is called the {\em $\binom{k}{s}$-sandwich}.

\begin{definition} Let $k$ be a non-negative integer and $s$ be a real number. The subsets
$$\mathbf 2^k_{<s}=\{(x_i)\in \mathbf 2^k:\textstyle{\sum\limits_{i=1}^kx_i}<s\}\mbox{ and }
\mathbf 2^k_{>s}=\{(x_i)\in \mathbf 2^k:\textstyle{\sum\limits_{i=1}^kx_i}>s\}$$
are called the {\em $s$-slices} of the $k$-cube $\mathbf 2^k$ where $\mathbf 2=\{0,1\}$ is the doubleton. For $s\in\{0,\dots,k\}$ the union of such slices has cardinality
$${|2^k_{<s}\cup 2^k_{>s}|=2^k-\binom{k}{s}=2^k-\frac{k!}{s!\,(k-s)!}}.$$
 
The subset
$$\Xiup^k_s=
\big(\{-1\}\times \mathbf 2^k_{<s}\big)\cup \big(\{0\}\times \mathbf 2^k_{<k}\big)\cup
\big(\{1\}\times\mathbf 2^k_{>s}\big)
$$
of the group $\IZ\times \IZ^{k}$ is called the {\em $\binom{k}{s}$-sandwich}.
For $s\in\{0,\dots,k\}$ it has cardinality
$$|\Xiup^k_s|=|\mathbf 2^k_{<k}|\cup|\mathbf 2^k_{<s}\cup \mathbf 2^k_{>s}|=2^{k+1}-1-\textstyle{\binom{k}{s}}.$$
\end{definition}  

The following theorem implies the upper bound in Theorem~\ref{t:bounds}(1). The proof of this theorem (given in Section~\ref{s:sandwich}) is not trivial and uses some elements of Algebraic Topology.

\begin{theorem}\label{sandwich} 
For every $k\in\IN$ and $s\le k-2$ the $\binom{k}{s}$-sandwich $\Xiup^{k}_{s}$ is a $(k+1)$-centerpole set in the group $\IZ\times\IZ^k$. 
\end{theorem}

In light of this theorem it is important to known the geometric structure of $\binom{k}{s}$-sandwiches $\Xiup^k_s$ for $s\le k-2$. For $k\le 3$ those sandwiches are written below:
\begin{itemize}
\item $\Xiup^0_{-2}=\{(1,0)\}$ is a singleton in $\IZ\times\IZ^0=\IZ\times\{0\}$;
\item $\Xiup^1_{-1}=\{(0,1),(1,0),(1,1)\}$ is the unit square without a vertex in $\IZ^2$; 
\item $\Xiup^2_0=\{(0,0,0),(0,0,1),(0,1,0),(1,0,1),(1,1,0),(1,1,1)\}$ is the unit cube without two opposite vertices in $\IZ^3$;
\item $\Xiup^3_0$ is the unit cube without two opposite vertices in $\IZ^4$, so $|\Xiup^3_0|=14$; 
\item $\Xiup^3_1$ is a 12-element subset in $\IZ^4$ whose slices $\{-1\}\times \bold 2^3_{<1}$, $\{0\}\times\bold 2^3_{<3}$, and $\{1\}\times\bold 2^3_{>1}$ have 1, 7, and 4 points, respectively. 
\end{itemize}

By a {\em triangle} ({\em centered at the origin}) we shall understand any affinely independent subset $\{a,b,c\}$ in $\IR^n$ (such that $a+b+c=0$).
A {\em tetrahedron} ({\em centered at the origin}) is any affinely independent subset $\{a,b,c,d\}\subset\IR^n$ (with $a+b+c+d=0$).

Let us observe that the sandwich
\begin{itemize}
\item $\Xiup^0_{-2}$ has cardinality $c_1(\IR^1)=1$ and is affinely equivalent to any singleton $\{a\}$ in $\IR^1$;
\item $\Xiup^1_{-1}$ has cardinality $c_2(\IR^2)=3$ and is affinely equivalent to any triangle $\Delta=\{a,b,c\}$ in $\IR^2$;
\item $\Xiup^2_0$ has cardinality $c_3(\IR^3)=6$ and is affinely equivalent to $\Delta\cup (x-\Delta)$ where $\Delta\subset\IR^3$ is a triangle centered at zero and $x\in \IR^3$ does not belong to the linear span of $\Delta$;
\item $\Xiup^3_1$ has cardinality $c_4(\IR^4)=12$ and is affinely equivalent to $(x-\Delta)\cup\Delta\cup(-x-\Delta)$ where $\Delta\subset\IR^4$ is a tetrahedron centered at zero and $x\in \IR^4$ does not belong to the linear span of $\Delta$.
\end{itemize} To see that $\Xiup^3_1$ is of this form, observe that $c=(\frac14,\frac12,\frac12,\frac12)$ is the barycenter of $\Xiup^3_1$ and $\Xiup^3_1-c=(x-\Delta)\cup\Delta\cup(-x-\Delta)$ for the tetrahedron
 $$\Delta=\{(0,0,0,1),(0,0,1,0),(0,1,0,0),(1,1,1,1)\}-c$$ and the point $x=(\frac12,0,0,0)$.

Now we briefly describe the structure of this paper. In Section~\ref{prepare} we establish a covering property of sandwiches, which will be essentially used in the proof of Theorem~\ref{sandwich}, given in Section~\ref{s:sandwich}.
Section~\ref{s:Tshape} is devoted to T-shaped sets which will give us lower bounds for the numbers $c_k^B(\IR^k)$. In Section~\ref{s:enlarge} we prove some lemmas that will help us to analyze the geometric structure of centerpole sets in Euclidean spaces. In Section~\ref{s:group-subgroup} we study the interplay between centerpole properties of subsets in a group and those of its subgroups. In Section~\ref{s:stab} we prove a particular case of the Stability Theorem~\ref{t:stab} for the groups $\IR^n\times\IZ^{m-n}$.
In Sections~\ref{s:t:bounds}, \ref{Pf:exact}, and \ref{s:Pf:lc} we give the proofs of Theorems~\ref{t:bounds}, \ref{t:exact}, and \ref{t:ckG}, respectively. Sections~\ref{s:p:rc} and \ref{s:t:stab} are devoted to the proofs of Proposition~\ref{p:rc} and Theorem~\ref{t:stab}. The final Section~\ref{s:OP} contains selected open problems.

\section{Covering $\Sigmaup_0$-sets by shifts of the sandwich $\Xiup^k_s$}
\label{prepare}

In this section we shall prove a crucial covering property of the $\binom{k}{s}$-sandwich $\Xiup^{k}_s$. In the next section this property will be used in the proof of Theorem~\ref{sandwich}. We assume that $k\in\w$ and $s\le k-2$ is integer. 

First we introduce the notion of a $\Sigmaup_0$-subset of the cube $\mathbf 2^{k+1}=\{0,1\}^{k+1}$.
For $i\in\{0,\dots,k\}$ consider the $i$-th coordinate projection 
$$\pr_i:\IR^{k+1}\to\IR,\;\;\pr_i:(x_j)_{j=0}^k\mapsto x_i.$$
The subsets of the form $\2^{k+1}\cap \pr_i^{-1}(l)$ for $l\in\{0,1\}$ are called the {\em facets} of the cube $\2^{k+1}$. 

Next, consider the function $$\Sigmaup:\IR^{k+1}\to\IR,\;\;\Sigmaup:(x_i)_{i=0}^k\mapsto\textstyle{\sum\limits_{i=1}^kx_i},$$
and observe that $\Sigmaup(\2^{k+1})=\{0,\dots,k\}$. 

Taking the diagonal product of the functions $\pr_0$ and $\Sigmaup$,  we obtain the linear operator $$\Sigmaup_0:\IR^{k+1}\to\IR^2,\;\;\Sigmaup_0:(x_i)_{i=0}^k\mapsto(x_0,\textstyle{\sum\limits_{i=1}^kx_i}).$$

\begin{definition}
A subset $\tau\subset \2^{k+1}$ will be called a {\em $\Sigmaup_0$-set}
if 
\begin{itemize}
\item $\tau$ lies in a facet of $\2^{k+1}$;
\item there exists $a\in\{0,\dots,k-1\}$ such that $\Sigmaup_0(\tau)\subset\{(0,a),(0,a+1),(1,a+1)\}$ or \newline
$\Sigmaup_0(\tau)\subset\{(0,a),(1,a),(1,a+1)\}$.
\end{itemize}
\end{definition}

\begin{lemma} \label{perebor} Each $\Sigmaup_0$-set $\tau\subset \bold 2^{k+1}$ is covered by a suitable shift $x+\Xiup^k_s$ of the $\binom{k}{s}$-sandwich $\Xiup^k_s$.
\end{lemma}

\begin{proof} 
Decompose the $\Sigmaup_0$-set $\tau$ into the union $\tau=\tau_0\cup\tau_1$ where $\tau_i=\tau\cap\pr_0^{-1}(i)$ for $i\in\{0,1\}$. 
By our hypothesis $\tau$ lies in a facet of the cube $\mathbf 2^{k+1}$. Consequently, there are numbers $\gamma\in\{0,\dots,k\}$ and $l\in\{0,1\}$ such that $\tau\subset\pr_\gamma^{-1}(l)$. If $\tau_0$ or $\tau_1$ is empty, then we can change the facet and assume that $\gamma=0$. 

Since $\tau$ is a $\Sigmaup_0$-set, the image $\Sigmaup_0(\tau)$ lies in one of the triangles: $\{(0,a),(0,a+1),(1,a+1)\}$ or $(\{(0,a),(1,a),(1,a+1)\}$ for some $a\in\{0,\dots,k-1\}$. 
This implies that $\Sigmaup(\tau)\subset\{a,a+1\}$. 

Identify the cube $\2^k$ with the subcube $\{0\}\times \2^k$ of $\Xiup^k_s$ and let $\bold e_0=(1,0,\dots,0)\in\2^{k+1}$. Then 
$$\Xiup^k_s=\2^k_{<k}\cup(\bold e_0+\2^k_{>s})\cup(-\bold e_0+\2^k_{<s}).$$
Depending on the value of $\gamma$, two cases are possible.
\smallskip

0. $\gamma=0$. This case has 4 subcases.
\smallskip

0.1. If $l=0$ and $a<k-1$ then $\Sigmaup_0(\tau)\subset\{(0,a),(0,a+1)\}\subset\{0,\dots,k-1\}$ and 
$\tau\subset \2^k_{<k}\subset\Xiup^k_s$.

0.2. If $l=0$ and $a\ge k-1$, then $a>k-2\ge s$  and  
$\tau\subset \2^k_{>s}\subset-\bold e_0+\Xiup^k_s$.

0.3. If $l=1$ and $a<k-1$, then $\Sigmaup_0(\tau)\subset \{(1,a),(1,a+1)\}\subset\{0,\dots,k-1\}$ and hence $\tau\subset\bold e_0+\2^k_{<k}\subset\bold e_0+\Xiup^k_s$.

0.4. If $l=1$ and $a\ge k-1$, then $a>k-2\ge s$ and then $\tau\subset \bold e_0+\2^k_{>s}\subset\Xiup^k_s$.
\smallskip

I. $\gamma\ne 0$. In this case $\tau_0$ and $\tau_1$ are not empty.
Let $\bold e_\gamma$ be the basic vector whose $\gamma$-th coordinate is 1 and the others are zero. By our assumption,  $\Sigmaup_0(\tau)\subset\{(0,a),(1,a),(1,a+1)\}$ or $\Sigmaup_0(\tau)\subset\{(0,a),(0,a+1),(1,a+1)\}$ for some $a\in\{0,\dots,k-1\}$.
So, we consider two subcases.
\smallskip

I.1. $\Sigmaup_0(\tau)\subset\{(0,a),(1,a),(1,a+1)\}$. 
This case has two subcases.
\smallskip

I.1.0. $l=0$. In this subcase $\Sigmaup(\tau)=\Sigmaup(\tau_0)\cup\Sigmaup(\tau_1)=\{a,a+1\}\subset\{0,\dots,k-1\}$ and hence $a\le k-2$. Depending on the value of $a$, we have three possibilities.

If $a>s$, then $\tau=\tau_0\cup\tau_1\subset \2^k_{<k}\cup (\bold e_0+\2^k_{>s})\subset\Xiup^k_s$.

If $a=s$, then for the shifted set $\bold e_\gamma+\tau$ we get
$$\Sigmaup_0(\bold e_\gamma+\tau)\subset\{(0,a+1),(1,a+1),(1,a+2)\}.$$ Since $a=s\le k-2$, we conclude that $\bold e_\gamma+\tau_0\subset \2^k_{<k}\subset\Xiup^k_s$. On the other hand, $\bold e_\gamma+\tau_1\subset \bold e_1+\2^k_{>s}\subset\Xiup^k_s$. Then $\tau\subset -\bold e_\gamma+\Xiup^k_s$.

If $a<s$, then $a+1\le s\le k-2$ and hence $\tau=\tau_0\cup\tau_1\subset \2^k_{<s}\cup(\bold e_0+\2^k_{<k})\subset \bold e_0+\Xiup^k_s$. 
 \smallskip

I.1.1. $l=1$. In this subcase three possibilities can occur:

If $a>s$, then $\tau=\tau_0\cup\tau_1\subset \2^k_{<k}+(\bold e_0+\2^k_{>s})\subset\Xiup^k_s$;

If $a<s$, then $a+1\le s\le k-2$ and then $\tau=\tau_0\cup\tau_1\subset \2^k_{<s}\cup(\bold e_0+\2^k_{<k})\subset \bold e_0+\Xiup^k_s$.

If $a=s$, then for the shift $-\bold e_\gamma+\tau$ we get $\Sigmaup_0({-}\bold e_\gamma{+}\tau)\subset\{(0,a{-}1),(1,a{-}1),(1,a)\}$ and hence
$-\bold e_\gamma+\tau\subset \2^k_{<s}\cup(\bold e_0+\2^k_{<k})\subset\bold e_0+\Xiup^k_s.$
Consequently, $\tau\subset \bold e_\gamma+\bold e_0+\Xiup^k_s$.
\smallskip

I.2. $\Sigmaup_0(\tau)\subset\{(0,a),(0,a+1),(1,a+1)\}$.  
Depending on the value of $l\in\{0,1\}$, consider two subcases.
\smallskip

I.2.0. $l=0$. In this case $\{0,\dots,k-1\}\supset\Sigmaup(\tau)=\Sigmaup(\tau_0)\cup\Sigmaup(\tau_1)=\{a,a+1\}\cup\{a+1\}$ and consequently, $a+1\le k-1$.
\smallskip

If $a\ge s$, then $\tau=\tau_0\cup\tau_1\subset \2^k_{<k}\cup (\bold e_0+\2^k_{>s})\subset\Xiup^k_s.$

If $a=s-1$, then we can consider the shift $\bold e_\gamma+\tau$ and repeating the preceding argument, show that $\bold e_\gamma+\tau\subset \Xiup^k_s$. Consequently, $\tau\subset-\bold e_\gamma+\Xiup^k_s$.

If $a<s-1$, then $\tau=\tau_0\cup\tau_1\subset \2^k_{<s}\cup (\bold e_0+\2^k_{<k})\subset \bold e_0+\Xiup^k_s$.
\smallskip

I.2.1. $l=1$. In this case we have four subcases.

If $a=k-1$, then for the shifted set $-\bold e_\gamma+\tau$ we get\newline $\Sigmaup_0(-\bold e_\gamma+\tau)\subset\{(0,a-1),(0,a),(1,a)\}$ and $-\bold e_\gamma+\tau\subset\2^k_{<k}\cup(\bold e_0+\2^k_{>s})=\Xiup^k_s$. Then $\tau\subset\bold e_\gamma+\Xiup^k_s$.

If  $s\le a<k-1$, then $\tau=\tau_0\cup\tau_1\subset \2^k_{<k}\cup(\bold e_0+\2^k_{>s})=\Xiup^k_s$.

If $a=s-1$, then for the shifted set $-\bold e_\gamma+\tau$ we get\newline 
$\Sigmaup_0(-\bold e_\gamma+\tau)\subset\{(0,a-1),(0,a),(1,a)\}$ and then
$-\bold e_\gamma+\tau\subset 2^k_{<s}\cup(\bold e_0+\2^k_{<k})=\bold e_0+\Xiup^k_s$
and $\tau\subset \bold e_\gamma+\bold e_0+\Xiup^k_s$.

If $a<s-1$, then $\tau=\tau_0\cup\tau_1\subset \2^k_{<s}\cup(\bold e_0+\2^k_{<k})=\bold e_0+\Xiup^k_s$.
\smallskip

This was the last of the 17 cases we have considered. 
\end{proof}

\section{Proof of Theorem~\ref{sandwich}}\label{s:sandwich}

The proof of Theorem~\ref{sandwich} uses the idea of the proof of Lemma 6 in \cite{Ba} (which established the upper bound $c_3(\IZ^3)\le 6$).

We need to prove that for every $k\le n$ and $s\le k-2$ the 
$\binom{k}{s}$-sandwich $\Xiup^k_s$ is $(k+1)$-centerpole in $\IZ\times \IZ^{k}=\IZ^{1+k}$. Assuming that this is not true, find a coloring $\chi:\IZ^{1+k}\to k+1=\{0,\dots,k\}$ such that $\IZ^{1+k}$ contains no unbounded monochromatic subset, symmetric with respect to some point $c\in\Xiup^k_s$. Observe that for each color $i\in\{0,\dots,k\}$ the intersection $A_i\cap (2c-A_i)$ is the largest subset of $A_i$, symmetric with respect to the point $c$. By our assumption, the (maximal $i$-colored $c$-symmetric) set $A_i\cap (2c-A_i)$ is bounded and so is the union $$B=\bigcup_{i=0}^k\,\bigcup_{c\in\Xiup^k_s}A_i\cap(2c-A_i)$$ of all such maximal symmetric monochromatic subsets. 

\begin{claimm}\label{cl1} $\chi(x)\notin\chi(-x+2\,\Xiup^k_s)$ for any $x\notin B$.
\end{claimm}

\begin{proof} Assuming conversely that $\chi(x)=\chi(-x+2c)$ for some $c\in\Xiup^k_s$, we get\break $\frac12(x+(-x+2c))=c$ and hence $x$ and $-x+2c$ are two points symmetric with respect to the center $c\in\Xiup^k_s$ and colored by the same color. Consequently, $x\in B$ by the definition of $B$.
\end{proof} 

Fix a number $n\in\IN$ so big that the cube $K=[-2n,2n]^{1+k}\subset\IR^{1+k}$  contains the bounded set $B$ in its interior and let $\partial K$ be the topological boundary $\partial K$ of the cube $K$ in $\IR^{1+k}$. Observe that Claim~\ref{cl1} implies:

\begin{claimm}\label{cl2} $\chi(-x)\notin\chi(x+2\,\Xiup^k_s)$ for each point $x\in\IZ^{1+k}\cap \partial K$.
\end{claimm}

We recall that for every $i\in k+1=\{0,\dots,k\}$  
$$\pr_i:\IR^{1+k}\to\IR,\;\;\pr_i:(x_j)_{j=0}^k\mapsto x_i,$$denotes the $i$th coordinate projection and $\bold e_i$ is the unit vector along the $i$-th coordinate axis, that is, $\pr_j(\bold e_i)=1$ if $i=j$, and 0 otherwise.

For a subset $J\subset \{0,\dots,k\}$ let $\bold e_J=\sum_{j\in J}\bold e_j\in\IR^{1+k}$ be the vector of the principal diagonal of the cube $\2^J=\{(x_i)_{i=0}^k\in \2^{1+k}:\forall\, i\notin J\;\; (x_i=0)\}\subset\2^{1+k}$. 

For a point $x\in\IR^{1+k}$ let $J_x=\{i\in k+1:x_i\notin 2\IZ\}$ and let ${\lfloor} x{\rfloor}$ be the unique point in $(2\IZ)^{1+k}$ such that $x \in \lfloor x\rfloor+2\cdot\2^{J_x}$. So, ${\lfloor}x{\rfloor}\le x\le{\lfloor}x{\rfloor}+2\,\bold e_{J_x}$.  

Consider the function $\Sigmaup:\IR^{k+1}\to\IR$ assigning to each sequence $x=(x_i)_{i=0}^k$ the sum $\Sigmaup(x)=\sum_{i=1}^kx_i$. The map $\Sigmaup$ combined with the 0th coordinate projection $\pr_0$ compose the linear operator $$\Sigmaup_0:\IR^{1+k}\to\IR^2,\;\;\Sigmaup_0:(x_i)_{i=0}^k\mapsto (x_0,\Sigmaup(x))=(x_0,\textstyle{\sum\limits_{i=1}^kx_i}).$$

Choose a triangulation $T$ of the boundary $\partial K$ of the cube $K=[-2n,2n]^{1+k}$ such that for each simplex $\tau$ of the triangulation there is 
a point $\dot\tau\in(2\IZ)^{1+k}$ such that $\frac12(\tau-\dot\tau)$ is a $\Sigmaup_0$-subset of $\2^{1+k}$. The reader can easily check that such a triangulation $T$ always exists. The choice of the triangulation $T$ combined with Lemma~\ref{perebor} implies

\begin{claimm}\label{cl3} Each simplex $\tau$ of the triangulation $T$ is covered by a suitable shift $x+2\Xiup^k_s$ of the homothetic copy $2\Xiup^k_s$ of the $\binom{k}{s}$-sandwich $\Xiup^k_s$.
\end{claimm}

Let $\Delta$ be (the geometric realization of) a simplex in $\IR^k$ with vertices $w_0,\dots,w_k$ such that $w_0+\dots +w_k=0$. The latter equality means that $\Delta$ is centered at the origin (which lies in the interior of $\Delta$). By $\Delta^{(0)}=\{w_0,\dots,w_k\}$ we denote the set of vertices of the simplex $\Delta$.

Each point $y\in\Delta$ can be uniquely written as the convex combination $y=\sum_{i=0}^k y_iw_i$ for some non-negative real numbers $y_0,\dots,y_k$ with $\sum_{i=0}^ky_i=1$. The set $$\supp(y)=\{i\in\{0,\dots,k\}:y_i\ne 0\}$$ is called the {\em support} of $y$. It is clear that $\supp(y)$ is the smallest subset of $\Delta^{(0)}$ whose convex hull contains the point $y$.

Identifying each number $i\in\{0,\dots,k\}$ with the vertex $w_i$ of $\Delta$, we can think of the coloring $\chi:\IZ^{1+k}\to \{0,\dots,k\}$ as a function $\chi:\IZ^{1+k}\to \Delta^{(0)}=\{w_0,\dots,w_k\}$.
 
Now extend the restriction $\chi|\partial K\cap(2\IZ)^{1+k}$ of $\chi$ to a simplicial map $f:\partial K\to \Delta$ (which is affine on the convex hull of each simplex $\tau\in T$). The simpliciality of $f$ implies:

\begin{claimm}\label{cl4} For each simplex $\tau\in T$ and a point $x\in\conv(\tau)$ $$\supp(f(x))\subset \chi(\tau)\subset\chi({\lfloor}x{\rfloor}+2\cdot \2^{J_x}).$$
\end{claimm} 

This Claim has the following corollary. 

\begin{claimm}\label{cl5} $f(\partial K)\subset\partial \Delta$.
\end{claimm}

\begin{proof} Given any point $x\in\partial K$, find a simplex $\tau\in T$ whose convex hull contains $x$. By the choice of the triangulation $T$ and Lemma~\ref{perebor}, $\tau\subset -y+2\Xiup^k_s$ for some point $y\in \IZ^{1+k}$. By Claim~\ref{cl2}, $\chi(-y)\notin\chi(\tau)$ and thus
$$f(x)\in\conv(f(\tau))=\conv(\chi(\tau))\subset\conv(\Delta^{(0)}\setminus\chi(-y))\subset\partial\Delta.$$
\end{proof}

Now consider the intersection $K_0=\{0\}\times[-2n,2n]^k$ of the cube $K$ with the hyperplane $\{0\}\times \IR^k$, which will be identified with the space $\IR^k$, and let  $\partial K_0=\partial K\cap \IR^k$ be the boundary of $K_0$.

For each subset $J\subset k+1=\{0,\dots,k\}$ consider the map $$p_J:\IR^{1+k}\to\IR,\;p_J:(x_i)_{i=0}^k\mapsto 1\cdot \prod_{j\in J}x_j.$$Here we assume that $p_\emptyset(x)=1$.
It follows that $\sum_{J\subset k+1}p_J(x)>0$ for all $x\in [0,2]^{k+1}$.

We remind that for a point $x\in\IR^{1+k}$, $J_x=\{i\in\{0,\dots,k\}:x_i\notin 2\IZ\}$ and ${\lfloor}x{\rfloor}$ stands for the unique point in $(2\IZ)^{1+k}$ such that $x \in {\lfloor}x{\rfloor}+\2^{J_x}$ where $\2^J=\{(x_i)_{i=0}^k\in\2^{k+1}:\forall i\notin J\;(x_i=0)\}$.  

Now consider the map $\varphi:\partial K_0\to \Delta$ defined by the formula:
$$\varphi(x)=\frac{\sum\limits_{J\subset k+1}p_J(x-{\lfloor}x{\rfloor})\cdot \chi({\lfloor}x{\rfloor}+\bold e_J)}{\sum\limits_{J\subset k+1}p_J(x-{\lfloor}x{\rfloor})}.$$
It can be shown that the map $\varphi$ is well-defined and continuous. 

\begin{claimm}\label{cl6} $\supp(\varphi(x))=\chi({\lfloor}x{\rfloor}+2\cdot \2^{J_x})\subset\chi(\lfloor x\rfloor+2\Xiup^k_s)$ for all $x\in \partial K_0$.
\end{claimm}

\begin{proof} Let $x\in\partial K_0$ be any point. The definition of $\varphi$ implies that $\supp(\varphi(x))=\chi({\lfloor}x{\rfloor}+\2^{J_x})$. The inclusion $x\in \partial K_0$ implies that the set $J_x=\{j\in\{0,\dots,k\}:\pr_j(x)\notin 2\IZ\}$ has cardinality $|J_x|<k$ and thus $\2^{J_x}\subset \{0\}\times \2^k_{<k}\subset\Xiup^k_s$. Consequently, $\lfloor x\rfloor+2\cdot\2^{J_x}\subset \lfloor x\rfloor +2\Xiup^k_s$ and $\chi(\lfloor x\rfloor+2\cdot\2^{J_x})\subset \chi(\lfloor x\rfloor +2\Xiup^k_s)$.
\end{proof}

\begin{claimm}\label{cl7} $\varphi(x)\ne \varphi(-x)$ for all $x\in \partial K_0$.
\end{claimm}

\begin{proof} Observe that $J_x=J_{-x}$ and $\lfloor -x\rfloor=-{\lfloor}x{\rfloor}-2\bold e_{J_x}$. By Claim~\ref{cl6},
$$\chi(-{\lfloor}x{\rfloor})=\chi(\lfloor -x\rfloor+2\cdot\bold e_{-J_x})\in\chi([-x]+2\cdot\2^{J_{-x}})=\supp(\varphi(-x)).$$
On the other hand, Claim~\ref{cl1} guarantees that
$$\chi(-{\lfloor}x{\rfloor})\not\ni\chi({\lfloor}x{\rfloor}+2\Xiup^k_s) \supset\chi({\lfloor}x{\rfloor}+2\cdot \2^{J_x})=\supp(\varphi(x)) 
.$$
Consequently, $\supp(\varphi(-x))\ne\supp(\varphi(x))$ and $\varphi(x)\ne \varphi(-x)$.
\end{proof} 

Finally, consider the homotopy $$(f_t):\partial K_0\times [0,1]\to\Delta,\quad f_t:x\mapsto t\varphi(x)+(1-t)f(x),$$
connecting the map $f=f_0$ with the map $\varphi=f_1$.

\begin{claimm}\label{cl8} $\supp(f_t(x))\subset\chi({\lfloor}x{\rfloor}+2\cdot\2^{J_x})\subset\partial\Delta$ for all $x\in\partial K_0$ and $t\in[0,1]$.
\end{claimm}

\begin{proof} The inclusion $\supp(f_t(x))\subset\chi({\lfloor}x{\rfloor}+2\cdot\2^{J_x})$ follows from Claims~\ref{cl4} and \ref{cl6}. 

The inclusion $x\in \partial K_0$ implies that the set $J_x=\{j\in\{0,\dots,k\}:\pr_j(x)\notin 2\IZ\}$ has cardinality $|J_x|<k$ and thus $\2^{J_x}\subset \{0\}\times \2^k_{<k}\subset\Xiup^k_s$. By Claim~\ref{cl1}, $\chi(-{\lfloor}x{\rfloor})\notin \chi({\lfloor}x{\rfloor}+2\Xiup^k_s)$ and then $$
\begin{aligned}
f_t(x)&\in\conv(\supp(f_t(x))\subset\conv(\chi({\lfloor}x{\rfloor}+2\cdot \2^{J_x}))\subset\\
&\subset\conv(\chi({\lfloor}x{\rfloor}+2\Xiup^k_s))\subset\conv(\Delta^{(0)}\setminus\chi({-}{\lfloor}x{\rfloor}))\subset\partial\Delta.
\end{aligned}
$$
\end{proof}

Let $S^{k-1}=\{x\in\IR^k:\|x\|=1\}$ be the unit sphere in $\IR^k$ with respect to the Euclidean norm $\|\cdot\|$ and $r:\IR^k\setminus\{0\}\to S^{k-1}$, $r:x\mapsto x/\|x\|$, be the radial retraction. Observe that its restriction $r|\partial \Delta$ to the  boundary of the geometric simplex $\Delta$ is a homeomorphism.
\smallskip

By Claim~\ref{cl5}, $f(\partial K)\subset\partial \Delta\subset\IR^k\setminus\{0\}$, so we can consider the map $g_0:\partial K\to S^{k-1}$ defined by $g_0(x)\mapsto r\circ f(x)=f(x)/\|f(x)\|$. By Claim~\ref{cl8}, the map $g_0|\partial K_0$ is homotopic to the map $$g_1:\partial K_0\to S^{k-1},\;\;g_1(x)\mapsto r\circ f_1(x)=r\circ\varphi(x).$$ 
It follows from Claim~\ref{cl7} that $g_1(x)\ne g_1(-x)$ for all $x\in \partial K_0$. 
This implies that the formula $$h_t(x)=\frac{g_1(x)-tg_1(-x)}{\|g_1(x)-tg_1(-x)\|},\quad x\in\partial K_0,\;t\in[0,1],$$ determines a well-defined homotopy $(h_t):\partial K_0\to S^{k-1}$ connecting the map $g_1$ with the map $$h_1(x)=\dfrac{g_1(x)-g_1(-x)}{\|g_1(x)-g_1(-x)\|},$$ which is antipodal in the sense that $h_1(-x)=-h_1(x)$. By \cite[Chap.4,\S7.10]{Spa}, each antipodal map between spheres of the same dimension is not homotopically trivial. Consequently, the antipodal map $h_1:\partial K_0\to S^{k-1}$ is not homotopically trivial. On the other hand, $h_1$ is homotopic to the map $h_0=g_1$, which is homotopic to  
$g_0|\partial K_0$ and the latter map is homotopically trivial since the boundary $\partial K_0$ of the cube $K_0$ is contractible in the boundary $\partial K$ of $K$. This contradiction completes the proof of Theorem~\ref{sandwich}.

\section{$T$-shaped sets in $\IR^n$}\label{s:Tshape}

Theorem~\ref{sandwich}, proved in the preceding section, yields an upper bound for the numbers $c_k(\IZ^k)$. A lower bound for the numbers $c_k^B(\IR^k)$ will be obtained by the technique of $T$-shaped sets created in \cite{Ba}.

Let $\IR_+=[0,\infty)$ be the closed half-line. For every $n\ge 0$ consider the subset $T_0\subset \IR^0$ defined inductively: $$T_0=\emptyset\subset\IR^0=\{0\},\;\;\;T_1=\{0\}\subset\IR^1, \mbox{ \ and \ }T_n=\big(\IR^{n-1}\times\{0\}\big)\cup \big(T_{n-1}\times\IR_+\big)\subset \IR^n$$ for $n>1$.

\begin{definition} A subset $C\subset \IR^n$ is called {\em $T$-shaped} if $f(C)\subset \IR\times T_{n-1}$ for some affine transformation $f:\IR^n\to\IR\times\IR^{n-1}$. 
The smallest cardinality of a subset $A\subset\IR^n$, which is not $T$-shaped is denoted by $t(\IR^n)$.
\end{definition} 

Let us describe the geometric structure of $T$-shaped sets.

We say that for $k\le n$, hyperplanes $H_1,\dots,H_k$ in $\IR^{n}$ are in {\em general position}
if they are pairwise distinct and their normal vectors are linearly independent. 
This happens if and only if there is an affine transformation $f:\IR^{n}\to\IR^{n}$ that maps the $i$-th hyperplane onto the hyperplane $\IR^{i-1}\times\{0\}\times\IR^{n-i}$ for all $i\in\{1,\dots,k\}$. 

We shall say that a hyperplane $H\subset \IR^n$ {\em does not separate} a subset $S\subset\IR^{n+1}$ if $S$ lies in one of two closed half-spaces bounded by the hyperplane $H$. Such a hyperplane $H$ will be called {\em non-separating} for $S$.
A hyperplane $H$ is called a {\em support hyperplane} for $S$ if $H\cap S\ne\emptyset$ and $H$ does not separate $S$.

\begin{proposition}\label{Tshape-char} Let $n\in\IN$. A subset $S\subset\IR^{n+1}$ is $T$-shaped if and only if $$S\subset H_1\cup \dots\cup H_n$$ for some hyperplanes $H_1,\dots,H_n$ in general position such that each hyperplane $H_i$, $1\le i\le n$, does not separate the set $S\setminus(H_1\cup\dots\cup H_{i-1})$.
\end{proposition}

\begin{proof} This proposition can be easily derived from the equality
$$\IR\times T_n=\bigcup_{i=0}^{n-1}\,\IR^{n-i}\times\{0\}\times\IR^i_+$$that can be easily proved by induction on $n$.
\end{proof}

By Lemma 7 of \cite{Ba}, $T$-shaped subsets of Euclidean spaces $\IR^k$ are $k$-centerpole for Borel colorings. Consequently, $t(\IR^n)\le c_n^B(\IR^n)$. This gives us a lower bound for the numbers $c_k^B(\IR^n)$ and $c_k(\IR^n)$:

\begin{proposition}\label{p1} $t(\IR^k)\le c_k^B(\IR^k)\le c_k^B(\IR^n)\le c_k(\IR^n)$ for any finite $k\le n$.
\end{proposition}

In the following theorem we collect all the available information on the numbers $t(\IR^n)$.

\begin{theorem}\label{t5}
 \begin{enumerate}
\item $t(\IR^1)=1$,
\item $t(\IR^2)=3$, 
\item $t(\IR^3)=6$, 
\item $t(\IR^4)=12$, 
\smallskip

\item $t(\IR^n)\le n^2-n+1$ for every $n\ge 1$;
\item $t(\IR^{n})\ge t(\IR^{n-1})+n+1$ for any $n\ge 4$,
\item $t(\IR^n)\ge \frac12(n^2+3n-4)$ for any $n\ge 4$.
\end{enumerate}
\end{theorem}

\begin{proof} 1. Since $T_0=\emptyset$, a subset of $\IR^1$ is $T$-shaped if and only if it is empty. Consequently, $t(\IR^1)=1$.
\smallskip

2. Since $T_1=\{0\}\subset\IR^1$, a subset $C\subset\IR^2$ is $T$-shaped if and only if $C$ lies in an affine line. Consequently, $t(\IR^2)=3$.
\smallskip

3. By Theorem~\ref{sandwich}, the 6-element $\binom{2}{0}$-sandwich $\Xiup^2_0$ is $3$-centerpole in $\IR^3$. Consequently, $c_3(\IR^3)\le 6$. By Proposition~\ref{p1}, $t(\IR^3)\le c_3(\IR^3)\le 6$. To see that $t(\IR^3)\ge 6$, we need to check that a subset $C\subset\IR^3$ of cardinality $|C|\le 5$ is $T$-shaped, which means that after a suitable affine transformation of $\IR^3$, $C$ can be embedded into $\IR\times T_2$. By the definition, $T_2=\IR\times\{0\}\cup \{0\}\times\IR_+$.

Consider the convex hull $\conv(C)$ of $C$ in $\IR^3$. If $C$ lies in an affine plane $H$, then applying to $\IR^3$ a suitable affine transformation, we can assume that $C\subset H=\IR\times\IR\times\{0\}\subset\IR\times T_2$. If $C$ does not lie in a plane, then the convex polyhedron $\conv(C)$ has a supporting plane $H_1$ such that $|H_1\cap C|\ge 3$. So, $C\setminus H_1$ lies in one of the closed half-spaces with respect to the plane $H_1$. Denote this subspace by $H_1^+$. The set $C\setminus H_1$ has cardinality $|C\setminus H_1|\le 2$ and hence it lies in an affine plane $H_2\subset \IR^3$ that meets $H_1$. Find an affine transformation $f:\IR^3\to\IR^3$ such that $f(H_1)=\IR\times\IR\times\{0\}$, $f(H_1^+)=\IR\times\IR\times\IR_+$ and $f(H_2)=\{\IR\}\times\{0\}\times\{\IR\}$. Then 
$$f(C)\subset \IR\times\IR\times\{0\}\cup\IR\times\{0\}\times \IR_+=\IR\times T_2$$and hence $C$ is $T$-shaped. 
\smallskip

4. By Theorem~\ref{sandwich}, the $\binom{3}{1}$-sandwich $\Xiup^3_1$ is $4$-centerpole in $\IZ^4$. Consequently, $$t(\IR^4)\le c_4(\IR^4)\le c_4(\IZ^4)\le|\Xiup^3_1|=2^4-1-\textstyle{\binom{3}{1}}=12.$$
The reverse inequality $t(\IR^4)\ge 12$ will be proved in Lemma~\ref{l11} below.
\smallskip

5. Let $C\subset\IR^n$ be a set consisting of $n^2-n+1=n(n-1)+1$ points in general position. This means that no $(n+1)$-element subset of $C$ lies in a hyperplane. 
Then $C$ can not be covered by less than $n$ hyperplanes and consequently $C$ is not $T$-shaped (because the set $\IR\times T_{n-1}$ lies in the union of $(n-1)$ hyperplanes). Then $t(\IR^n)\le|C|=n^2-n+1$.
\smallskip

6. First we prove the inequality 
\begin{equation}\label{eq7}t(\IR^n)\ge \min\{2t(\IR^{n-1}),t(\IR^{n-1})+n+1\}
\end{equation} for every $n\ge 2$. Take any subset $C\subset\IR^n$ of cardinality\newline $|C|<\min\{2t(\IR^{n-1}),t(\IR^{n-1})+n+1\}$. We need to show that $C$ is $T$-shaped. 

Consider the convex hull $\conv(C)$ of $C$ in $\IR^n$. If $\conv(C)$ lies in some hyperplane, then $C$ is $T$-shaped by the definition. So, we assume that $\conv(C)$ does not lie in a hyperplane and then $\conv(C)$ is a compact convex body in $\IR^n$. Let $H$ be a supporting hyperplane of $\conv(C)$ having maximal possible cardinality of the intersection $C\cap H$. It is clear that $|C\cap H|\ge n$.
\smallskip

Now two cases are possible:
\smallskip

a) The set $C\setminus H$ lies in a hyperplane $H_1$, parallel to $H$. Then $H_1$ is a supporting hyperplane of $\conv(C)$ and then $|C\cap H_1|\le|C\cap H|$ by the choice of $H$. Now we see that $|C\cap H_1|\le\frac12|C|<t(\IR^{n-1})$.

Applying to $\IR^n=\IR^{n-1}\times \IR$ a suitable affine transformation, we can assume that $H=\IR^{n-1}\times \{0\}$ and $C\setminus H\subset \IR^{n-1}\times \IR_+$. Let $\pr:\IR^n\to\IR^{n-1}$ be the coordinate projection. Since $|\pr_n(C\setminus H)|<t(\IR^{n-1})$, the set $C'=\pr_n(C\setminus H)$ is $T$-shaped. This means that there is an affine transformation $f:\IR^{n-1}\to\IR^{n-1}$ such that $f(C')\subset \IR\times T_{n-2}$.  This affine transformation $f$ induces the affine transformation 
$$\Phi:\IR^{n-1}\times\IR\to\IR^{n-1}\times\IR,\;\;\Phi(x,y)=(f(x),y),$$ such that $$\Phi(C)=\Phi(C\cap H)\cup\Phi(C\setminus H)\subset \big(\IR\times\IR^{n-2}\times\{0\}\big)\cup \big(\IR\times T_{n-2}\times \IR_+\big)=\IR\times T_{n-1}.$$The affine transformation $\Phi$ witnesses that the set $C$ is $T$-shaped.
\smallskip

b) The set $C\setminus H$ does not lie in a hyperplane parallel to $H$. Then $C\setminus H$ contains two distinct points $x,y$ such that the vector $\vec{xy}$ is not parallel to $H$. Applying to $\IR^n=\IR^{n-1}\times \IR$ a suitable affine transformation, we can assume that $H=\IR^{n-1}\times \{0\}$, $C\setminus H\subset \IR^{n-1}\times\IR_+$, and under the projection $\pr:\IR^{n-1}\times\IR\to\IR^{n-1}$ the images of the points $x$ and $y$ coincide.
Then the projection $C'=\pr(C\setminus H)$ has cardinality $|C'|\le|C\setminus H|-1<|C|-|C\cap H|-1<t(\IR^{n-1})+n+1-n-1=t(\IR^{n-1})$. Continuing as in the preceding case, we can find an affine transformation $\Phi$, witnessing that $C$ is a $T$-shaped set in $\IR^n$.
\smallskip

This proves the inequality (\ref{eq7}). By analogy we can prove that $t(\IR^n)\ge t(\IR^{n-1})+n$. Since $t(\IR^1)=1$, by induction we can show that $t(\IR^n)\ge \frac12 n(n+1)$. In particular, $t(\IR^{n-1})\ge \frac12n(n-1)\ge n+1$ for all $n\ge 4$. In this case
$$t(\IR^n)\ge \min\{2t(\IR^{n-1}),t(\IR^{n-1})+n+1\}=t(\IR^{n-1})+n+1.$$
\smallskip

7. The lower bound $t(\IR^n)\ge \frac12(n^2+3n-4)$, $n\ge 4$, will be proved by induction. For $n=4$ it is true according to the statement (4). Assuming that it is true for some $n>4$ and applying the lower bound (6), we get
$$t(\IR^{n+1})\ge t(\IR^n)+(n+1)+1\ge \frac12(n^2+3n-4)+n+2=\frac12((n+1)^2+3(n+1)-4).$$ 
\end{proof}

To finish the proof of Theorem~\ref{t5}, it remains to prove the promised:

\begin{lemma}\label{l11} Each subset $C\subset\IR^4$ of cardinality $|C|<12$ is $T$-shaped.
\end{lemma}

\begin{proof} Assume that some subset $C\subset\IR^4$ of cardinality $|C|<12$ is not $T$-shaped. Without loss of generality, $|C|=11$. 
 
We recall that a hyperplane $H\subset\IR^4$ is called a {\em support hyperplane} for $C$ if $C\cap H\ne \emptyset$ and $H$ does not separate $C$ (which means that $C$ lies in a closed half-space $H^+$ bounded by the hyperplane).

\begin{claimm}\label{s1} Each support hyperplane $H\subset\IR^4$ for $C$ has at most 5 common points with $C$.
\end{claimm}

\begin{proof} Assume that $H$ is a support hyperplane for $C$ with $|H\cap C|>5$. After a suitable affine transformation of $\IR^4$, we can assume that 
$H=\IR^3\times\{0\}$ and $C\subset \IR^3\times\IR_+$. Let $\pr:\IR^4\to\IR^3$ be the coordinate projection.
Since $|C\setminus H|=|C|-|C\cap H|<11-5=6$ and $t(\IR^3)=6$ (by Theorem~\ref{t5}(3)), $\pr(C\setminus H)$ is $T$-shaped in $H$ and so $C$ is $T$-shaped $\IR^4$.  
\end{proof}

\begin{claimm}\label{s2} For any two parallel hyperplanes $H_1$ and $H_2$ in $\IR^4$ the set $C\setminus (H_1\cup H_2)$ is non-empty. 
\end{claimm}

\begin{proof}
Otherwise one of these hyperplanes contains more than 6 points, which contradicts Claim~\ref{s1}.
\end{proof}

\begin{claimm}\label{s3} Each support hyperplane $H$ for the set $C$ has less than 5 common points with $C$.
\end{claimm}

\begin{proof}
Previous claim guarantees the existence of two distinct points $a,b\in C$ that lie in an affine line $L$ that meets $H$. After a suitable affine transformation of $\IR^4$, we can assume that $H=\IR^{3}\times \{0\}$, $C\subset \IR^{3}\times\IR_+$, and $L=\{0\}^3\times\IR$. Let $\pr:\IR^4\to\IR^3$ be the coordinate projection. 
Assuming that $|H\cap C|\ge 5$ and taking into account that $\pr(a)=\pr(b)$, we conclude that $$|\pr(C\setminus H)|\le |C\setminus H|-1=|C|-|C\cap H|-1\le 5<6=t(\IR^3).$$ It follows that $\pr(C\setminus H)$ is $T$-shaped in $\IR^3$ and then $C$ is $T$-shaped in $\IR^4$.
\end{proof}

The characterization of $T$-shaped sets given in Proposition~\ref{Tshape-char} implies:

\begin{claimm}\label{s6} If $H_1$ is a support hyperplane for $C$, $H_2$ is a support hyperplane for $C\setminus H_1$ and $H_1,H_2$ are not parallel, then  $|C\setminus (H_1\cup H_2)|\ge 3$ and if $|C\setminus (H_1\cup H_2)|=3$, then the set $C\setminus (H_1\cup H_2)$ does not lie in a line but lies in a plane, parallel to $H_1\cap H_2$.
\end{claimm}

\begin{claimm}\label{s7}If $H_1$ and $P_2$ are parallel support hyperplanes for $C$ and $|H_1\cap C|=4$, then  $|P_2\cap C|=1$.
\end{claimm}

\begin{proof} By Claim~\ref{s3}, $C\setminus H_1$ does not lie in a hyperplane. Now consider 4 cases.
\smallskip

1) $|P_2\cap C|>4$. In this case $C$ is $T$-shaped by Claim~\ref{s3}.
\smallskip

2) $|P_2\cap C|=4$. We claim that the set $P_2\cap C$ does not lie in a plane $P$.
Otherwise $P$ can be enlarged to a support hyperplane that contains $\ge 5$ points of $C$, which is forbidden by Claim~\ref{s3}. Therefore, the convex hull of $P_2\cap C$ is a convex body in $P_2$ and we can find a support hyperplane $H_2$ for $C\setminus H_1$ that meets $H_1$, has at least 4 common points with $C\setminus H_1$ and exactly three common points with the set $C\cap P_2$. In this case the unique point $c_2$ of the set $C\cap P_2\setminus H_2$ lies in $C\setminus (H_1\cup H_2)$. By Proposition~\ref{Tshape-char}, the set $C\setminus (H_1\cup H_2)$ contains exactly 3 points that lie in a plane parallel to $H_1\cap H_2$. Since this set contains the point $c_2\in C\cap P_2$, we conclude that $C\setminus (H_1\cup H_2)\subset P_2$ and hence $|C\cap P_2|=6$, which is a contradiction.  
\smallskip

3) $|P_2\cap C|=3$. Let $Pl$ be a plane which contains $P_2\cap C$ and lies in the hyperplane $P_2$. We claim that the set $C\setminus(H_1\cup Pl)$ lies in a plane $Pl_1$ that is parallel to $Pl$. Let $S$ be the set of all points $x\in C\setminus (H_1\cup Pl)$ that belong to a support hyperplane $H_x$ to $C\setminus H_1$ that has at least 4 common points with $C\setminus H_1$ and contains the plane $Pl$. Claim~\ref{s6} guarantees that the set $C\setminus (H_1\cup H_x)$ contains exactly 3 elements and lies in a plane that is parallel to the intersection $H_1\cap H_x$ (which is parallel to $Pl$). Since the set $C\setminus H_1$ does not lie in a hyperplane, the set $S$ contains more that one point, which implies that the set $C\setminus (H_1\cup Pl)=\bigcup_{x\in S}C\setminus(H_1\cup H_x)$ lies in a plane $Pl_1$ that is parallel to the plane $Pl$. Let $H_2$ be the hyperplane that contains the parallel planes $Pl$ and $Pl_1$. Since $H_2$ meets $H_1$, we see that $C\subset H_1\cup H_2$ is $T$-shaped by Proposition~\ref{Tshape-char} and this is a contradiction.
\smallskip

4) $|P_2\cap C|=2$. Since $C\setminus H_1$ does not lie in a hyperplane, there is a support hyperplane $H_2$ to $C\setminus H_1$ such that $|H_2\cap (C\setminus H_1)|\ge 4$ and $|H_2\cap P_2\cap C|=1$. It follows that the hyperplane $H_2$ does not coincide with $P_2$ and hence meets the hyperplane $H_1$. By Claim~\ref{s6}, the complement $C\setminus (H_1\cup H_2)$ contains exactly 3 points that lie in a plane, parallel to $H_1\cap H_2$. Since $C\setminus(H_1\cup H_2)$ meets the hyperplane $P_2$ we conclude that $C\setminus(H_1\cup H_2)\subset P_2$ and $|C\cap P_2|\ge 4$, which is a contradiction. 
\end{proof}

\begin{claimm}\label{s7a} If $P_1$ and $P_2$ are parallel support hyperplanes for $C$ and $|P_1\cap C|=4$, then the set $C\setminus(P_1\cup P_2)$ lies in a hyperplane $P_3$ that is parallel to $P_1$ and $P_2$.
\end{claimm}

\begin{proof} By Claim~\ref{s7}, $|P_2\cap C|=1$ and hence $|C\setminus(P_1\cup P_2)|=6$. Let $x$ be the unique point of $P_2\cap C$. Take any support hyperplane $H\ni x$ for the set $C\setminus P_1$ such that $|H\cap C|\ge 4$. Since $H$ meets $P_1$, Proposition~\ref{Tshape-char} guarantees that the set $C'=C\setminus (P_1\cup H)$ contains exactly 3-points that lie in a plane parallel to the intersection $P_1\cap H$ and hence parallel to $P_1$. The hyperplane $H'$ containing the set $C'\cup\{x\}$ is a support hyperplane for the set $C\setminus P_1$. Applying Proposition~\ref{Tshape-char}, we conclude that the set $C''=C\setminus (P_1\cup H')=C\cap H\setminus P_2$ contains exactly 3 points lying in a plane parallel to $P_1\cap H'$.  Thus $C\setminus (P_1\cup P_2)$ lies in two planes parallel to $P_1$ and hence it lies in a hyperplane $P_3$. Proposition~\ref{Tshape-char} implies that the hyperplane $P_3$ is parallel to $P_1$.
\end{proof}

By an {\em octahedron} in a linear space $L$ we understand a set of the form $$c+\{\bold e_i,-\bold e_i:1\le i\le 3\}$$where  $\bold e_1,\bold e_2,\bold e_3$ are linearly independent vectors in $L$ and $c\in L$ is the {\em center} of the octahedron. Up to an affine equivalence an octahedron is a unique 6-element set $X$with 3-dimensional affine hull $A$ such that for each support plane $P\subset A$ of $X$ with $|P\cap X|\ge 3$ the set $X\setminus P$ contains 3 points and lies in a plane $P'$, parallel to $P$.

\begin{claimm}\label{s8} If $P_1$ and $P_2$ are parallel support hyperplanes for $X$ and $|P_1\cap C|=4$, then the set $C\setminus(P_1\cup P_2)$ is an octahedron that lies in a hyperplane $P_3$, parallel to $P_1$.
\end{claimm}

\begin{proof} By the preceding Claim, the set $K=C\setminus (P_1\cup P_2)$ lies in a hyperplane $P_3$, parallel to $P_1$. Let us show that $K$ does not lie in a plane. In the opposite case, we could find a hyperplane $H_2$ that contains the set $K$ and meets the hyperplane $P_1$. Then for each hyperplane $H_3$ that contains the unique point $C\cap P_2$ and has one-dimensional intersection with $P_1\cap H_2$, we get $C\subset P_1\cup H_2\cup H_3$ witnessing that $C$ is $T$-shaped.

Thus the affine hull of $K$ is 3-dimensional. To see that $K$ is an octahedron, it suffices to check that for each support plane $P\subset P_3$ of $K$ with $|P\cap K|\ge 3$ the set $K\setminus P$ contains exactly 3 points and lies in a plane parallel to $P$.

Let $x$ be the unique point of the set $C\cap P_2$ and $H_2$ be the hyperplane containing the plane $P$ and passing through $x$. It follows that $H_2$ is a support hyperplane for the set $C\setminus P_1$. By Claim~\ref{s6}, the set $C\setminus (P_1\cup H_2)=K\setminus P$ contains exactly 3 elements and lies in a plane $P'$ parallel to the intersection $H_1\cap H_2$. 

Now let $H_2'$ be the hyperplane that contains the support plane $P'$ and passes through the point $x$. Since $P'$ is a support plane for $K$ in the hyperplane $P_3$, $H_3$ is a support hyperplane for $K\cup\{x\}=C\setminus P_1$ in $\IR^4$. 
Since $H_3'$ intersects $P_1$, Claim~\ref{s6} guarantees that the set $C\setminus(P_1\cup H_2')=K\setminus P'$ contains exactly 3 points and the plane $P$ containing these 3 points is parallel to $P_1\cap H_2'$ which is parallel to the plane $P'$.
\end{proof} 

After this preparatory work we are ready to finish the proof of Lemma~\ref{l11}.
As $C$ is not $T$-shaped, it does not lie in a hyperplane. So, we can find a support hyperplane $P_1$ for $C$ such that $|P_1\cap C|\ge 4$. Let $P_2$ be a support hyperplane for $C$, which is parallel to $P_1$. By Claim~\ref{s7}, $|P_1\cap C|=4$ and $|P_2\cap C|=1$. Let $p_2$ be the unique point of the set $P_2\cap C$. By Claim~\ref{s8}, $C\setminus(P_1\cup P_2)$ is an octahedron that lies in a hyperplane $P_3$, parallel to the hyperplanes $P_1$ and $P_2$. Let $c$ be the center of this octahedron and $2c-p_2$ be the point, symmetric to $p_2$ with respect to $c$.

Fix any 
3-element subset $F$ of $P_1\cap C$  such that $2c-p_2\in F$ if $2c-p_2\in C\cap P_1$. Next, find a hyperplane $H_1$ for $C$ that contains $F$ and meets $C\setminus H_1$ at some point $a$. 
If $a=p_2$, then the set $C\subset H_1\cup P_3\cup (C\cap P_1\setminus F)$ is $T$-shaped by Proposition~\ref{Tshape-char}.

Consequently, $a$ is a point of the octahedron $C\cap P_3$ with center $c$. Let $H_2$ be a support hyperplane for $C$ that is parallel to the hyperplane $H_1$. By Claims~\ref{s7} and \ref{s8}, $|C\cap H_1|=4$, $|C\cap H_2|=1$ and $C\setminus(H_1\cup H_2)$ is an octahedron that lies in a hyperplane $H_3$, parallel to $H_1$ and $H_2$. 
If $H_3$ does not meet the octahedron $C\cap P_3$, then $(C\cap P_3)\cap (C\cap H_3)=(C\cap P_3)\setminus H_1=C\cap P_3\setminus\{a\}$. In this case the octahedra $C\cap P_3$ and $C\cap H_3$ have 5 common points and hence lie in the same hyperplane $P_3=H_3$, which is not possible. So, the support hyperplane $H_3$ meets the octahedron $C\cap P_3$ at a single point and this point is $2c-a$. In this case the octahedra $C\cap P_3$ and $C\cap H_3$ have 4 common points which belong to the set $C\cap P_3\setminus\{a,2c-a\}$ and lie in the 2-dimensional plane $P_3\cap H_3$. This implies that the octahedra $C\cap P_3$ and $C\cap H_3$ have the common center $c$. Since $p_2\in C\cap H_3$, the point $2c-p_2$ belongs to the octahedron $C\cap H_3\subset C$. It follows from $p_2\in P_2$ and $c\in P_3$ that $2c-p_2\in C\setminus(P_2\cup P_3)=C\cap P_1$ and hence $2c-p_2\subset F\subset H_1$ by the choice of the set $F$. On the other hand, $2c-p_2$ belongs to the hyperplane $H_3$, which is disjoint with $H_1$ and this is a desired contradiction.
\end{proof}

\section{Enlarging non-centerpole sets}\label{s:enlarge}

In this section we prove several lemmas on enlarging non-centerpole subsets. Namely, we show that under certain conditions, a non-$k$-centerpole subset  $C$ of a topological group $X$ (possibly enlarged by one or two points) remains not $k$-centerpole in the direct sum $X\oplus\IR$. The group $X\oplus\IR$ can be identified with the direct product $X\times\IR$ so that $X$ is identified with the subgroup $X\times\{0\}\subset X\times\IR$ while the real line $\IR$ is identified with the subgroup $\{e\}\times\IR\subset X\times\IR$ where $e$ is the neutral element of the group $X$.

\begin{lemma}\label{l:+0} If for $k\ge 2$ a subset $C\subset X$ of a topological group $X$ is not $k$-centerpole (for Borel colorings), then set $C$ is not $k$-centerpole in $X\oplus\IR$.
\end{lemma}

\begin{proof}
Since the set $C\subset X$ is not $k$-centerpole (for Borel colorings), there exists a (Borel) coloring $\chi:X\to k$ such that $X$ contains no monochromatic unbounded subset, which is symmetric with respect to a point $c\in C$. Extend $\chi$ to a (Borel) coloring $\tilde \chi:X\times\IR\to k$ letting 
$$\tilde\chi(x,t)=\begin{cases}
\chi(x)&\mbox{if $t=0$}\\
0&\mbox{if $t<0$}\\
1&\mbox{if $t>0$}.
\end{cases} 
$$This coloring witnesses that $C$ is not $k$-centerpole in $X\oplus\IR$ (for Borel colorings).
\end{proof}

\begin{lemma}\label{l:+1} If for $k\ge 3$ a subset $C\subset X$ of a topological group $X$ with $c_2^B(X)\ge 2$ is not $k$-centerpole (for Borel colorings), then for each $x\in X\times(0,\infty)$ the set $C\cup\{x\}$ is not $k$-centerpole for (Borel) colorings of the topological group $X\oplus\IR$.
\end{lemma}

\begin{proof} 
Without loss of generality we may assume that $x=(e,1)$ where $e$ is the neutral element of topological group $X$. Fix a (Borel) coloring $\chi:X\to k$ witnessing that the subset $C\subset X$ is not $k$-centerpole (for Borel colorings). 

This coloring induces a (Borel) 2-coloring $\chi_2:X\to 2$ defined by
$$\chi_2(x)=\min\big(\{0,1\}\setminus\chi(x^{-1})\big)\mbox{ \ for $x\in X$}.$$
 
Since $c_2^B(X)\ge 2$,  there exists a Borel coloring $\chi_1: X\to 2$ witnessing that the singleton $\{e\}$ is not 2-centerpole for Borel colorings of $X$.

It is easy to see that the (Borel) coloring $\tilde \chi:X\times\IR\to k$ defined by
$$ \tilde \chi(x,t)=\begin{cases}
\chi(x),&\mbox{if $t=0$},\\
\chi_1(x),&\mbox{if $t=1$},\\
\chi_2(x),&\mbox{if $t=2$},\\
0,&\mbox{if $1<t\neq 2$},\\
1,&\mbox{if $0<t<1$},\\
2&\mbox{if $t<0$}
\end{cases}
$$witnesses that the set $C\cup\{(e,1)\}$ fails to be $k$-centerpole for (Borel) colorings of the topological group $X\oplus\IR$.
\end{proof}

\begin{lemma}\label{l:c3=6} $c_3^B(\IR^m)\ge 6$ for all $m\ge 3$.
\end{lemma}

\begin{proof} By Theorem~\ref{t5}(3) and Proposition~\ref{p1}, $c_3^B(\IR^3)\ge t(\IR^3)=6$.

Next, we check that $c_3^B(\IR^4)\ge 6$. Assuming that $c_3^B(\IR^4)<6$ find a subset $C\subset \IR^4$ of cardinality $|C|\le 5$, which is 3-centerpole for Borel colorings of $\IR^4$.

Since $|C|\le 5$, there is a 3-dimensional hyperplane $H_3\subset\IR^4$ such that $|C\setminus H_3|\le 1$. Since $|C\cap H_3|\le|C|<6=c_3^B(\IR^3)$, the set $C\cap H_3$ is not 3-centerpole for Borel colorings of $H_3$. By (the proof of) Proposition 4.1 of \cite{BDR}, $c_2^B(\IR^3)=3\ge 2$. By  Lemma~\ref{l:+1}, the set $C$ is not 3-centerpole for Borel colorings of $H_3\oplus \IR$ (which can be identified with $\IR^4$).

Now assume that the inequality $c_3^B(\IR^{m-1})\ge 6$ has been proved for some $m\ge 4$. Assuming that $c_3^B(\IR^m)\le 5$ find a subset $C\subset \IR^{m}$ of cardinality $|C|\le 5$ which is 3-centerpole for Borel colorings of $\IR^{m}$. This set lies in a $m-1$ dimensional hyperplane and according to Lemma~\ref{l:+0}, is 3-centerpole for Borel colorings of $\IR^{m-1}$. Then   $c_3^B(\IR^{m-1})\le|C|\le 5$, which contradicts the inductive assumption.
\end{proof}

\begin{lemma}\label{l:+2} If for $k\ge 4$ a subset $C\subset X$ of a topological group $X$ with $c_2^B(X)\ge 3$ is not $k$-centerpole (for Borel colorings), then for any 2-element set $A\subset X\times(0,\infty)$ the set $C\cup A$ is not $k$-centerpole for (Borel) colorings of the topological group $X\oplus\IR$.
\end{lemma}

\begin{proof} Let $(a,v)$ and $(b,w)$ be the points of the 2-element set $A\subset X\times(0,\infty)$. We can assume that $v\le w$. Let $\chi_0:X\to k$ be a (Borel) coloring witnessing that the set $C$ is not $k$-centerpole for (Borel) colorings of the group $X$.

Consider the Borel 4-coloring 
$\psi:\IR\to 4$ of the real line defined by
$$\psi(t)=\begin{cases}
3&\mbox{if $t\le 0$}\\
0&\mbox{if $0<t\le v$}\\
1&\mbox{if $v<t\le w$}\\
2&\mbox{if $w<t$}
\end{cases}
$$
and observe that for each $c\in\{0,v,w\}$ and $t\in\IR\setminus\{c\}$ we get $\psi(t)\ne\psi(2c-t)$.

We consider 2 cases.
\smallskip

1) $v=w$. In this case we can assume that $v=w=1$. Since $c_2^B(X)\ge 3$, there exists a Borel coloring $\chi_1: X\to 2$ witnessing that the 2-element set $\{a,b\}\subset X$ is not 2-centerpole for Borel colorings of $X$.
The (Borel) coloring $\chi_0$ induces the (Borel) coloring $\chi_2: X\to 3$ defined by the formula 
$$\chi_2(x)=\min \big(\{0,1,2\}\setminus \{\chi_0(a x^{-1}a),
\chi_0(bx^{-1}b)\}\big).$$

Now we see that the (Borel) coloring $\tilde\chi:X\times\IR\to k$ defined by
$$ \tilde \chi(x,t)=\begin{cases}
\chi_t(x),&\mbox{if $t\in\{0,1,2\}$},\\
\psi(t),&\mbox{otherwise}
\end{cases}
$$
witnesses that the set $C\cup A$ is not $k$-centerpole for (Borel) colorings of the topological group $X\oplus\IR$.
\smallskip

2) The second case occurs when $v\ne w$. Without loss of generality, $v<w$ and $w-v=1$. This case has three subcases.
\smallskip

2a) $v=1$ and $w=2$. In this case we can assume that $b=e$ is the neutral element of the group $X$.

Since $c_2^B(X)\ge 3$, there is a Borel 2-coloring $\chi_1:X\to 2$ witnessing that the singleton $\{a\}$ is not 2-centerpole in $X$. By the same reason, there is a Borel 2-coloring $\phi:X\to 2$ witnessing that the singleton $\{b\}=\{e\}$ is not 2-centerpole for Borel colorings of $X$.
Using the colorings $\phi$ and $\chi_0$ one can define a (Borel) 3-coloring $\chi_2:X\to 3$ such that $\chi_2(x)\ne \chi_0(ax^{-1}a)$ for all $x\in X$ and $\chi_2(x)\ne \chi_2(x^{-1})$ if and only if $\phi(x)\ne\phi(x^{-1})$.

Such a coloring $\chi_2:X\to 3$ can be defined by the formula
$$
\chi_2(x)=\begin{cases}
\min\big(3\setminus\{\chi_0(axa),\chi_0(ax^{-1}a)\}\big),&\mbox{if $\phi(x)=\phi(x^{-1})$};\\
\phi(x),
&\mbox{if $\chi_0(ax^{-1}a){\ne}\phi(x){\ne}\phi(x^{-1}){\ne}\chi_0(axa)$};\\
\min\big(3\setminus\{\phi(x^{-1}),\chi_0(ax^{-1}a)\}\big),
&\mbox{if $\chi_0(ax^{-1}a){=}\phi(x){\ne}\phi(x^{-1}){\ne}\chi_0(axa)$};\\
\phi(x),
&\mbox{if $\chi_0(ax^{-1}a){\ne}\phi(x){\ne}\phi(x^{-1}){=}\chi_0(axa)$};\\
\phi(x^{-1}),
&\mbox{if $\chi_0(ax^{-1}a){=}\phi(x){\ne}\phi(x^{-1}){=}\chi_0(axa)$}.
\end{cases}
$$

Let $\chi_3:X\to 2$ be the Borel 2-coloring defined by $\chi_3(x)=1-\chi_1(x^{-1})$ for $x\in X$. It is clear that $\chi_3(x^{-1})\ne \chi_1(x)$ for all $x\in X$.
Finally, consider the Borel 2-coloring $\chi_4:X\to 2$ defined by 
$$\chi_4(x)=\min\big(\{0,1\}\setminus\{\chi_0(x^{-1})\}\big)\mbox{ \ for $x\in X$}.$$

The (Borel) colorings $\psi,\chi_0,\chi_1,\chi_2,\chi_3,\chi_4$ compose a (Borel) $k$-coloring
$\tilde\chi:X\times\IR\to k$,
$$\tilde\chi(x,t)=
\begin{cases}
\chi_t(x),&\mbox{if $t\in\{0,1,2,3,4\}$},\\
\psi(t),&\mbox{otherwise},\\
\end{cases}
$$
witnessing that the set $C\cup A$ is not $k$-centerpole for (Borel) colorings of $X\oplus \IR$.
\smallskip

2b) $v=2$ and $w=3$. Since $c_2^B(X)\ge 3>1$, there is a Borel 2-coloring $\chi_2:X\to 2$ witnessing that the singleton $\{a\}$ is not 2-centerpole for Borel colorings of $X$. By the same reason, there is a Borel 2-coloring $\chi_3:X\to 2$ witnessing that the singleton $\{b\}$ is not 2-centerpole for Borel colorings of $X$. 

Next consider the (Borel) colorings $\chi_1:X\to 2$, $\chi_4:X\to 3$, and $\chi_6:X\to 2$ defined by the formulas
$$
\begin{aligned}
\chi_1(x)&=1-\chi_3(ax^{-1}a),\\
\chi_4(x)&=\min\big(3\setminus\{\chi_0(ax^{-1}a),\chi_2(bx^{-1}b)\}\big),\\
\chi_6(x)&=\min\big(2\setminus\{\chi_0(bx^{-1}b)\}).
\end{aligned}$$

The (Borel) colorings $\psi$ and $\chi_t$, $t\in\{0,1,2,3,4,6\}$, compose the (Borel) coloring
$\tilde\chi:X\times\IR\to k$ defined by
$$\tilde\chi(x,t)=\begin{cases}
\chi_t(x),&\mbox{if $t\in\{0,1,2,3,4,6\}$},\\
\psi(t),&\mbox{otherwise}.
\end{cases}
$$
This coloring $\tilde \chi$ witnesses that the set $C\cup A$ is not $k$-centerpole for (Borel) colorings of $X\oplus\IR$.
\smallskip

2c) $v\notin\{1,2\}$. Since $c_2^B(X)>1$ there is a Borel 2-coloring $\chi_v:X\to 2$ witnessing that the singleton $\{a\}$ is not 2-centerpole for Borel colorings of $X$. By the same reason, there is a Borel 2-coloring $\chi_w:X\to \{1,2\}$ witnessing that the singleton $\{b\}$ is not 2-centerpole for Borel colorings of $X$.

Next, define the (Borel) colorings $\chi_{2v},\chi_{2w}:X\to 3$ by the formula
$$\chi_{2v}(x)=\min\big(3\setminus\{\chi_0(ax^{-1}a),\psi(2)\}\big)\mbox{ \ and \ }
\chi_{2w}(x)=\min\big(2\setminus\{\chi_0(bx^{-1}b)\}\big).$$
Here let us note that the points $2v$ and $2$ are symmetric with respect to $w$ in the group $\IR$. 
 
Finally, define a (Borel) $k$-coloring $\tilde\chi:X\oplus\IR\to k$ letting
$$\tilde\chi(x,t)=\begin{cases}
\chi_t(x)&\mbox{if $t\in\{0,v,w,2v,2w\}$}\\
\psi(t)&\mbox{otherwise}.
\end{cases}
$$This coloring witnesses that the set $C\cup A$ is not $k$-centerpole for (Borel) colorings of the topological group $X\oplus\IR$.  
\end{proof}

\begin{lemma}\label{l:c4>8} $c_4^B(\IR^m)\ge 8$ for all $m\ge 4$.
\end{lemma}

\begin{proof} This lemma will be proved by induction on $m\ge 4$. For $m=4$ the inequality $c_4^B(\IR^4)\ge t(\IR^4)=12\ge 8$ follows from Lemma~\ref{l11}.
Assume that for some $m\ge 4$ we know that $c_4^B(\IR^m)\ge 8$. The inequality $c_4^B(\IR^{m+1})\ge 8$ will follow as soon as we check that each 7-element subset $C\subset\IR^{m+1}$ is not 4-centerpole for Borel colorings of $\IR^{m+1}$. 

Given a 7-element subset $C\subset\IR^{m+1}$, find a support $m$-dimensional hyperplane $H\subset\IR^{m+1}$ that has at least $\min\{m+1,|C|\}\ge 5$ common points with the set $C$. After a suitable shift, we can assume that the intersection $C\cap H$ contains the origin of $\IR^{m+1}$. In this case $H$ is a linear subspace of $\IR^{m+1}$ and $\IR^{m+1}$ can be written as the direct sum $\IR^{m+1}=H\oplus \IR$.

Since $|H\cap C|\le|C|\le7$, the inductive assumption guarantees that $H\cap C$ is not 4-centerpole for Borel colorings of $H$. By Lemma~\ref{l:c3=6}, $c_3^B(\IR^m)\ge 3$. Since $|C\setminus H|\le 2$, we can apply Lemma~\ref{l:+2} and conclude that $C$ is not 4-centerpole for Borel colorings of the topological group $H\oplus\IR=\IR^{m+1}$.
\end{proof}

\section{centerpole sets in subgroups and groups}\label{s:group-subgroup}

It is clear that each $k$-centerpole subset $C\subset H$ in a subgroup $H$ of a topological group $G$ is $k$-centerpole in $G$. In some cases the converse statement also is true.

\begin{lemma}\label{l:sub} If a subset $C$ of an abelian topological group $G$ is $k$-centerpole in $G$ for some $k\ge 2$, then it is $k$-centerpole in the subgroup $H=\langle C\rangle+G[2]$. 
\end{lemma}

\begin{proof} Observe that for each $x\in G\setminus H$ the cosets $c+2\langle C\rangle$ and $-x+2\langle C\rangle$ are disjoint. Assuming the opposite, we would conclude that $2x\in 2\langle C\rangle$ and hence $x\in\langle C\rangle+G[2]=H$, which contradicts the choice of $x$.  

Now we are able to prove that the set $C$ is $k$-centerpole in the group $H$. Given any $k$-coloring $\chi:H\to k$, extend $\chi$ to a $k$-coloring $\tilde \chi:G\to k$ such that for each $x\in G\setminus H$ the coset $x+2\langle C\rangle$ is monochromatic and its color is different from the color of the coset $-x+2\langle C\rangle$.

Since $C$ is $k$-centerpole in the group $G$, there is an unbounded monochromatic subset $S\subset G$ such that $S=2c-S$ for some $c\in C$. We claim that $S\subset H$. Assuming the converse, we would find a point $x\in S\setminus H$ and conclude that the coset $x+2\langle C\rangle$ has the same color as the coset $2c-x+2\langle C\rangle=-x+2\langle C\rangle$, which contradicts the choice of the coloring $\tilde\chi$.
\end{proof}

The Borel version of this result is a bit more difficult.

\begin{lemma}\label{l:subB} Let $k\ge 2$ and $H$ be a Borel subgroup of an abelian topological group $G$ such that $G[2]\subset H$. A subset $C\subset H$ is $k$-centerpole for Borel colorings of $H$ if $C$ is $k$-centerpole for Borel colorings of $G$, the subgroup $2H=\{2x:x\in H\}$ is closed in $G$, and  the subspace $X=(G/2H)\setminus (H/2H)$ contains a Borel subset $B$ that has one-point intersection with each set $\{x,-x\}$, $x\in X$. Such a Borel set $B\subset X$ exists if the space $X$ is paracompact.
\end{lemma}

\begin{proof}  Given any Borel $k$-coloring $\chi:H\to k$, extend $\chi$ to a Borel $k$-coloring $\tilde \chi:G\to k$ defined by
$$\tilde \chi(x)=\begin{cases}
\chi(x),&\mbox{if $x\in H$},\\
0,&\mbox{if $x\in G\setminus H$ and $x+2H\in B$},\\
1,&\mbox{if $x\in G\setminus H$ and $x+2H\notin B$}
\end{cases}
$$
Since $C$ is $k$-centerpole for Borel colorings of the group $G$, there is an unbounded monochromatic subset $S\subset G$, symmetric with respect to some point $c\in C$. We claim that $S\subset H$, witnessing that $C$ is $k$-centerpole for Borel colorings of $H$. 

Assuming conversely that $S\not\subset H$, find a point $x\in S\setminus H$. It follows that $x$ and $2c-x$ have the same color. If this color is 0, then the cosets $x+2H$ and $2c-x+2H=-x+2H=-(x+2H)$ both belong to the set $B\subset G/2H$. By our hypothesis $B$ has one-point intersection with the set $\{x+2H,-(x+2H)\}$. Consequently, $x+2H=-(x+2H)$ and hence $2x\in 2H$ and $x\in H+G[2]=H$, which contradicts the choice of the point $x$. If the color of the cosets $x+2H$ and $2c-x+2H=-(x+2H)$ is 1, then $(x+2H),-(x+2H)\notin B$ and then $x+2H=-(x+2H)$ because $B$ has one-point intersection with the set $\{x+2H,-(x+2H)\}$. This again leads to a contradiction. 

\begin{claimm} If the space $X=(G/2H)\setminus (H/2H)$ is paracompact, then $X$ contains a Borel subset $B\subset X$ that has one-point intersection with each set $\{x,-x\}$, $x\in X$.
\end{claimm}

Consider the action
$$\alpha:C_2\times X\to X,\quad\alpha:(\e,x)\mapsto\e\cdot x,$$
of the cyclic group $C_2=\{1,-1\}$ on the space $X$ and let $X/C_2=\big\{\{x,-x\}:x\in X\big\}$ be the orbit space of this action. It is easy to check that the orbit map $q:X\to X/C_2$ is closed and then the orbit space $X/C_2$ is paracompact as the image of a paracompact space under a closed map, see Michael Theorem 5.1.33 in \cite{En}.
 
Since $H\supset 2H+G[2]$, for every $x\in G\setminus H$ the cosets $x+2H$ and $-x+2H$ are disjoint, which implies that each point $x\in X$ is distinct from $-x$.
Then each point $x\in X$ has a neighborhood $U_x\subset X$ such that $U_x\cap -U_x=\emptyset$. Replacing $U_x$ by $U_x\cap(-U_{-x})$ we can additionally assume that $U_x=-U_{-x}$. Now consider the open neighborhood $U_{\pm x}=q(U_x)=q(U_{-x})\subset X/C_2$ of the orbit $\{x,-x\}\in X/C_2$ of the point $x\in X$. By the paracompactness of $X/C_2$ the open cover $\{U_{\pm x}:x\in X\}$ of $X/C_2$ has a $\Sigmaup$-discrete refinement $\U=\bigcup_{n\in\w}\U_n$. This means that each family $\U_n$, $n\in\w$, is discrete in $X/C_2$. For each $U\in\U$ find a point $x_U\in X$ such that $U\subset U_{\pm x_U}$. For every $n\in\w$ consider the open subset $W_n=\bigcup_{U\in\U_n}q^{-1}(U)\cap U_{x_U}$ of the space $X$ and let $\pm W_n=-W_n\cup W_n$. One can check that the Borel subset  $$B=\bigcup_{n\in\w}(W_n\setminus\bigcup_{i<n}\pm W_i)$$ of  $X$ 
 has one-point intersection 
 with each orbit $\{x,-x\}$, $x\in X$. 
\end{proof}

The following lemma will be helpful in the proof of the upper bound $rc^B_k(G)\le c_k^B(G)-2$ from Proposition~\ref{p:rc}. 

\begin{lemma}\label{l:afhull} Let $k\ge 4$ and $C\subset\IR^\w$ be a finite $k$-centerpole subset for Borel colorings of $\IR^\w$. Then the affine hull of $C$ in $\IR^\w$ has dimension $\le |C|-3$.
\end{lemma}

\begin{proof} This lemma will be proved by induction on the cardinality $|C|$.

First observe that $|C|\ge c_k^B(\IR^\w)\ge c_3^B(\IR^\w)\ge 6$ by Lemma~\ref{l:c3=6}. So, we start the induction with $|C|=6$. 

Suppose that either $m=6$ or $m>6$ and the lemma is true for all $C$ with $6\le|C|<m$. Fix a $k$-centerpole subset $C\subset\IR^\w$ for Borel colorings of cardinality $|C|=m$. We need to show that the affine hull $A$ of $C$ has dimension $\dim A\le m-3$. Assuming the opposite, we can find a support hyperplane $H\subset A$ for $C$ such that $|H\cap C|\ge \dim H+1=\dim A\ge|C|-2$ and hence $0<|C\setminus H|\le 2$.
 After a suitable shift, we can assume that $H$ contains the origin of $\IR^\w$ and hence is a subgroup of $\IR^\w$. In this case the affine hull $A$ is a linear subspace in $\IR^\w$ that can be identified with the direct sum $H\oplus\IR$. It follows that $\dim H=\dim A-1\ge |C|-2-1\ge|C\cap H|-2$.

We claim that the set $H\cap C$ is not $k$-centerpole for Borel colorings of the topological group $H$. 

If $6\le |C\cap H|<|C|=m$, then by the inductive assumption, the set $C\cap H$ is not $k$-centerpole for Borel colorings of $\IR^\w$ because its affine hull $H$ has dimension $\dim H\ge |C\cap H|-2$. If $|C\cap H|<6$ (which happens for $m=6$), then 
the inequalities
$c_k^B(H)\ge c_3^B(H)\ge 6=m=|C|>|H\cap C|$ given by Lemma~\ref{l:c3=6} guarantee that $C\cap H$ is not $k$-centerpole for Borel colorings of $\IR^\w$.

By (the proof) of Proposition 1 in \cite{BDR}, $c_2^B(H)=3$. Since $H$ is a support hyperplane for $C$ and $|C\setminus H|\le 2$, we can apply Lemma~\ref{l:+2} and conclude that $C$ is not $k$-centerpole for Borel colorings of $H\oplus \IR=A$.
Since the subgroup $2A$ is closed in the metrizable group $\IR^\w$, by Lemma~\ref{l:subB}, $C$ is not $k$-centerpole for Borel colorings of $\IR^\w$ and this is a desired contradiction that completes the proof of the inductive step and base of the induction.
\end{proof}

\section{Stability Properties}\label{s:stab}

In this section we shall prove some particular cases of the Stability Theorem~
\ref{t:stab}.

\begin{lemma}\label{l1:stab} For any numbers $k\ge 2$ and $n\le m$
$$c_k^B(\IR^n\times\IZ^{m-n})=\begin{cases}
c_k^B(\IR^n\times \IZ^{\w}),&\mbox{if $m\ge rc_k^B(\IR^n\times\IZ^\w)$,}\\
c_k^B(\IR^\w),&\mbox{if $n\ge rc_k^B(\IR^\w)$.}
\end{cases}
$$
\end{lemma}

\begin{proof} First assume that $m\ge rc_k^B(\IR^n\times\IZ^\w)$. By the definition of the number $r=rc_k^B(\IR^n\times\IZ^\w)$, the topological group $G=\IR^n\times \IZ^\w$ contains a $k$-centerpole subset $C\subset G$ of cardinality $|C|=c_k^B(G)$ that generates a subgroup $\langle C\rangle\subset \IZ^\w$ of $\IZ$-rank $r$. 
It follows that the linear subspace $L\subset \IR^n\times\IR^\w$ generated by the set $C$ has dimension $r$. Then $H=L\cap G$, being a closed 
subgroup of $\IZ$-rank $r$ in the $r$-dimensional vector space $L$ is topologically isomorphic to $\IR^s\times\IZ^{r-s}$ for some $s\le r\le m$, see Theorem 6 in \cite{Morris}. Taking into account that $H$ is a closed subgroup of $G=\IR^n\times\IZ^\w$, we conclude that $s\le n$. By Lemma~\ref{l:subB}, the set $C$ is $k$-centerpole in $H$ for Borel colorings. Consequently,
$$c_k^B(\IR^n\times\IZ^\w){\le} c_k^B(\IR^n\times \IZ^{m-n}){\le} c_k^B(\IR^s\times\IZ^{r-s}){=}c_k^B(H){\le} |C|{=}c_k^B(G){=}c_k^B(\IR^n\times\IZ^\w)$$implies 
 the desired equality $c_k^B(\IR^n\times\IZ^{m-n})=c_k^B(\IR^n\times\IZ^\w)$.
\smallskip

Now assume that $n\ge rc_k^B(\IR^\w)$. In this case we can repeat the above argument for a set $C\subset \IR^\w$ of cardinality $|C|=c_k^B(\IR^\w)$ that generates a subgroup $\langle C\rangle \subset\IR^\w$ of $\IZ$-rank $r=rc_k^B(\IR^\w)$. Then the linear subspace $L\subset\IR^\w$ generated by the set $C$ is topologically isomorphic to $\IR^r$. By Lemma~\ref{l:subB}, the set  $C$ is $k$-centerpole for Borel colorings of $L$. Since $\IR^r\hookrightarrow \IR^n\times\IZ^{m-n}\hookrightarrow\IR^\w$, we get
$$c_k^B(\IR^\w)\le c_k^B(\IR^n\times\IZ^{m-n})\le c_k^B(\IR^r)=c_k^B(L)\le|C|=c_k^B(\IR^\w)$$and hence $c_k^B(\IR^n\times\IZ^{m-n})=c_k^B(\IR^\w)$. 
\end{proof}

\begin{lemma}\label{l2:stab} $c_k(\IR^n\times\IZ^{m-n})=c_k^B(\IZ^\w)$ for any numbers $k\in\IN$ and $n\le m$ with $m\ge c_k^B(\IZ^\w)$.
\end{lemma}

\begin{proof} For $k=1$ the equality $c_k(\IR^n\times\IZ^{m-n})=1=c_k^B(\IZ^\w)$ is trivial. So we assume that $k\ge 2$. 

We claim that $c_k^B(\IZ^\w)\le c_k(\IR^m)$. Indeed, take any $k$-centerpole subset $C\subset\IR^\w$ of cardinality $|C|=c_k(\IR^m)$. By Lemma~\ref{l:sub}, the set $C$ is $k$-centerpole in the subgroup $\langle C\rangle\subset \IR^\w$ generated by $C$. 
Being a torsion-free finitely-generated abelian group, $\langle C\rangle$ is algebraically isomorphic to $\IZ^r$ for some $r\in\w$. Then 
$$c_k(\IZ^r)\le c_k(\langle C\rangle)\le|C|=c_k(\IR^m).$$ 

On the other hand, Lemma~\ref{l1:stab} ensures that
$$c_k(\IR^m)\le c_k(\IZ^m)=c_k^B(\IZ^m)=c_k^B(\IZ^\w).$$
Unifying these inequalities  we get
$$c_k^B(\IZ^\w){\le} c_k^B(\IZ^r){=}c_k(\IZ^r){\le} c_k(\IR^m){\le} c_k(\IR^n\times\IZ^{m-n}){\le} c_k(\IZ^m){=}c_k^B(\IZ^m){=}c_k^B(\IZ^\w),$$ which implies the desired equality $c_k(\IR^n\times\IZ^{m-n})=c_k^B(\IZ^\w)$.
\end{proof}

\section{Proof of Theorem~\ref{t:bounds}}\label{s:t:bounds}

1. The upper bound $c_k(\IZ^n)\le c_k(\IZ^k)\le 2^k-1-\max_{s\le k-2}\binom{k-1}{s-1}$ for $k\le n$ follows from Theorem~\ref{sandwich}.

2. By Proposition~\ref{p1} and Theorem~\ref{t5}(7), 
$c_n(\IZ^n)\ge c_n(\IR^n)\ge c_n^B(\IR^n)\ge t(\IR^n)\ge\frac12(n^2+3n-4).$
\smallskip

For technical reasons, first we prove the statement (4) of Theorem~\ref{t:bounds} and after that return back to the statement (3).
\smallskip

4. Let $1\le k\le m\le\w$ be two numbers. We need to prove that $c_k^B(\IR^m)<c_{k+1}^B(\IR^{m+1})$ and  $c_k(\IR^m)<c_{k+1}(\IR^{m+1})$. 

First we assume that $m$ is finite. The strict inequality $c_k^B(\IR^m)<c_{k+1}^B(\IR^{m+1})$ will follow as soon as we show that any subset $C\subset\IR^{m+1}$ of cardinality $|C|\le c_k^B(\IR^{m})$ fails to be $(k+1)$-centerpole for Borel colorings of $\IR^{m+1}$. If $C$ is a singleton, then it is not $(k+1)$-centerpole since $c_{k+1}^B(\IR^{m+1})\ge c_2^B(\IR^{m+1})\ge 3$ by (the proof of) Proposition 4.1 in \cite{BDR}. So, $C$ contains two distinct points $a,b$. Let $L=\IR\cdot (a-b)\subset\IR^{m+1}$ be the linear subspace generated by the vector $a-b$.  Write the space $\IR^{m+1}$ as the direct sum $\IR^{m+1}=H\oplus L$ where $H$ is a linear $m$-dimensional subspace of $\IR^{m+1}$ and consider the projection $\pr:\IR^{m+1}\to H$ whose kernel is equal to $L$. Since $\pr(a)=\pr(b)$, the projection of the set $C$ onto the subspace $H$ has cardinality
$|\pr(C)|<|C|\le c_k^B(\IR^m)=c_k^B(H)$ and hence $\pr_H(C)$ is not $k$-centerpole for Borel $k$-colorings of the group $H$. Consequently, there is a Borel $k$-coloring $\chi:H\to k$ such no monochromatic unbounded subset of $H$ is symmetric with respect to a point $c\in\pr(C)$.

 For a real number $\gamma\in\IR$, consider the half-line $L^+_\gamma=\{t(a-b):t\ge\gamma\}$ of $L$. Since the subset $C\subset \IR^{m+1}$ is finite, there is $\gamma\in\IR$ such that $C\subset H+L^+_\gamma$.

Now define a Borel $(k+1)$-coloring $\tilde\chi:H\oplus L\to k+1=\{0,\dots,k\}$ by the formula
$$\tilde\chi(x)=\begin{cases}
\chi(\pr(x)),&\mbox{if $x\in H+L_\gamma^+$},\\
k,&\mbox{otherwise}.
\end{cases}
$$
It can be shown that this coloring witnesses that $C$ is not $(k+1)$-centerpole for Borel colorings of $\IR^{m+1}=H\oplus L$.

Now assume that the number $m$ is infinite. Then for the finite number $r=\max\{rc_k^B(\IR^\w),rc_{k+1}^B(\IR^\w)\}$ we get $c_k^B(\IR^r)=c_k^B(\IR^\w)$ and $c_{k+1}^B(\IR^{r+1})=c_{k+1}^B(\IR^\w)$ by the stabilization Lemma~\ref{l1:stab}. Since $r$ is finite, the case considered above guarantees that
$$c_k^B(\IR^m)=c_k^B(\IR^m)=c_k^B(\IR^r)<c_{k+1}^B(\IR^{r+1})=c_{k+1}^B(\IR^\w)=c_{k+1}(\IR^{m+1}).$$

By analogy we can prove the strict inequality $c_k(\IR^m)<c_k(\IR^{m+1})$. 
\smallskip

3. Now we are able to prove the lower bound $c_k^B(\IR^\w)\ge k+4$ from the statement (3) of Theorem~\ref{t:bounds}. By the preceding item, $c^B_{k+1}(\IR^\w)\ge 1+c_{k}^B(\IR^\w)$ for all $k\in\IN$. By induction, we shall show that $c_k^B(\IR^\w)\ge k+4$ for all $k\ge 4$. For $k=4$ the inequality $c_4^B(\IR^\w)\ge 8\ge 4+4$ was proved in Lemma~\ref{l:c4>8}. Assuming that $c_k^B(\IR^\w)\ge k+4$ for some $k\ge 4$, we conclude that $c_{k+1}^B(\IR^\w)>c_k^B(\IR^\w)\ge k+4$ and hence $c_{k+1}^B(\IR^\w)\ge (k+1)+4$.

Now we see that for every $n\ge k\ge 4$ we have the desired lower bound:
$$c_k^B(\IR^n)\ge c_k^B(\IR^\w)\ge k+4.$$

5. Let $k\in\IN$ and $n,m\in\w\cup\{\w\}$ be numbers with $1\le k\le n+m$. We need to prove that $c_k^B(\IR^n\times \IZ^m)<c_{k+1}^B(\IR^n\times\IZ^{m+1})$ and  $c_k(\IR^n\times \IZ^m)<c_{k+1}(\IR^n\times\IZ^{m+1})$. According to the Stabilization Lemma~\ref{l1:stab}, it suffices to consider the case of finite numbers $n,m$.

First we prove the inequality $c_k^B(\IR^n\times \IZ^m)<c_{k+1}^B(\IR^n\times\IZ^{m+1})$.
We need to show that each subset $C\subset \IR^n\times\IZ^{m+1}$ of cardinality $|C|\le c_k^B(\IR^n\times\IZ^m)$ is not $(k+1)$-centerpole in $\IR^n\times\IZ^{m+1}$ for Borel colorings. We shall identify $\IR^n\times\IZ^{m+1}$ with the direct sum $\IR^n\oplus\IZ^{m+1}$. Since $k\le n+m$, Theorem~\ref{sandwich} implies that the numbers $|C|\le c_k^B(\IR^n\times\IZ^m)\le c_k(\IZ^{n+m})\le c_k(\IZ^k)$ all are finite.
\smallskip

Three cases are possible:
\smallskip

(i) $|C|\le 1$. In this case we can assume that $C=\{0\}$ and take any coloring $\chi:\IR^n\oplus\IZ^{m+1}\to k+1$ such that the color of each non-zero element 
$x\in\IR^n\times\IZ^{m+1}$ differs from the color of $-x$. This coloring witnesses that $C$ is not $(k+1)$-centerpole in $\IR^n\times\IZ^{m+1}$.
\smallskip

(ii) $|C|>1$ and $C\subset z+\IR^n$ for some $z\in\IZ^{m+1}$. Without lose of generality, $z=0$ and hence $C\subset \IR^n$. Take two distinct points $a,b\in C$ and consider the 1-dimensional linear subspace $L=\IR\cdot(a-b)\subset\IR^n$ generated by the vector $a-b$. Write the space $\IR^n$ as the direct sum $\IR^n=L\oplus H$ where $H$ is a linear $(n-1)$-dimensional subspace of $\IR^n$ and consider the projection $\pr:\IR^n\oplus\IZ^{m+1}\to H\oplus\IZ^{m+1}$ whose kernel is equal to $L$. Since $\pr(a)=\pr(b)$, the projection of the set $C$ onto the subgroup $H\oplus\IZ^{m+1}$ of $\IR^n\oplus\IZ^{m+1}$ has cardinality
$$|\pr(C)|<|C|\le c_k^B(\IR^n\times\IZ^m)\le c_k^B(\IR^{n-1}\times\IZ^{m+1})=c_k^B(H\oplus\IZ^{m+1})$$ and hence $\pr_H(C)$ is not $k$-centerpole for Borel colorings of the group $H\oplus \IZ^{m+1}$. Consequently, there is a Borel $k$-coloring $\chi:H\oplus\IZ^{m+1}\to k$ such no monochromatic unbounded subset of $H\oplus\IZ^{m+1}$ is symmetric with respect to a point $c\in\pr(C)$.

 For a real number $\gamma\in\IR$, consider the half-line $L^+_\gamma=\{t(a-b):t\ge\gamma\}$ of $L$. Since the subset $C\subset \IR^n\oplus\IZ^{m+1}=H\oplus L\oplus\IZ^{m+1}$ is finite, there is $\gamma\in\IR$ such that $C\subset H+L^+_\gamma+\IZ^{m+1}$.

Now define a Borel $(k+1)$-coloring $\tilde\chi:H\oplus L\oplus\IZ^{m+1}\to k+1=\{0,\dots,k\}$ by the formula
$$\tilde\chi(x)=\begin{cases}
\chi(\pr(x)),&\mbox{if $x\in H+L_\gamma^++\IZ^{m+1}$},\\
k,&\mbox{otherwise}.
\end{cases}
$$
It can be shown that this coloring witnesses that $C$ is not $(k+1)$-centerpole for Borel colorings of $\IR^n\oplus\IZ^{m+1}=H\oplus L\oplus\IZ^{m+1}$. 
\smallskip

(iii) The set $C\subset\IR^n\oplus\IZ^{m+1}$ contains two points $a,b$ whose projections on the subspace $\IZ^{m+1}$ are distinct. Without loss of generality, the projections of $a,b$ on the last coordinate are distinct. Then the 1-dimensional subspace $L=\IR\cdot(a-b)$ of $\IR^n\times\IR^{m+1}$ meets the subspace $\IR^n\oplus \IR^m$ and hence $\IR^n\oplus\IR^{m+1}$ can be identified with the direct sum $\IR^n\oplus\IR^m\oplus L$. 
 Let $\pr:\IR^n\times\IR^{m+1}\to\IR^n\times\IR^m$ be the projection whose kernel coincides with $L$. Since $\pr$ is an open map, the image $H=\pr(\IR^n\times\IZ^{m+1})$ is a locally compact (and hence closed) subgroup of $\IR^n\times\IR^m$, which can be written as the countable union of shifted copies of the space $\IR^n$. By Theorem 6 of \cite{Morris}, $H$ is topologically isomorphic to $\IR^n\times\IZ^m$. 
It follows from the definition of $H$ that $\IR^n\oplus\IZ^{m+1}\subset H\oplus L$.

Since $\pr(a)=\pr(b)$, the projection of the set $C$ has cardinality $|\pr(C)|<|C|\le c_k^B(\IR^n\oplus\IZ^m)=c_k^B(H)$, which means that $\pr(C)$ is not $k$-centerpole for Borel colorings of $H$.
Consequently, there is a Borel $k$-coloring $\chi:H\to k$ such no monochromatic unbounded subset of $H$ is symmetric with respect to a point $c\in\pr(C)$.

 For a real number $\gamma\in\IR$, consider the half-line $L^+_\gamma=\{t(a-b):t\ge\gamma\}$ of $L$. Since the subset $C\subset H\oplus L$ is finite, there is $\gamma\in\IR$ such that $C\subset H+L^+_\gamma$.

Now define a Borel $(k+1)$-coloring $\tilde\chi:H\oplus L\to k+1$ by the formula
$$\tilde\chi(x)=\begin{cases}
\chi(\pr(x)),&\mbox{if $x\in H+L_\gamma^+$},\\
k,&\mbox{otherwise}.
\end{cases}
$$
It can be shown that this coloring witnesses that $C$ is not $(k+1)$-centerpole for Borel colorings of $H\oplus L\supset \IR^n\oplus\IZ^{m+1}$. 
\smallskip

After considering these three cases, we can conclude that $c_{k+1}^B(\IR^n\times\IZ^{m+1})>c_k^B(\IR^n\times\IZ^m)$.
\smallskip

Deleting the adjective ``Borel'' from the above proof, we get the proof of the strict inequality    $$c_{k}(\IR^n\times\IZ^{m})<c_{k+1}(\IR^n\times\IZ^{m+1}).$$

\section{Proof of Theorem~\ref{t:exact}}\label{Pf:exact}

In this section we prove Theorem~\ref{t:exact}.
Let $k,n,m$ are cardinals.
 We shall use known upper bounds for the numbers $c_k(\IZ^n)$, lower bounds for $t(\IR^n)$ and the inequality
$$t(\IR^{n+m})\le c_k^B(\IR^{n+m})\le c_k^B(\IR^n\times\IZ^m)\le c_k(\IR^n\times\IZ^m)\le c_k(\IZ^m)$$
established in Proposition~\ref{p1}.
\smallskip

1. Assume that $n+m\ge 1$. Since each singleton is 1-centerpole for (Borel) colorings of the group $\IR^n\times\IZ^m$, we conclude that $c_1(\IR^n\times\IZ^m)=c_1^B(\IR^n\times\IZ^m)=1$.
\smallskip

2. Assume that $n+m\ge 2$. The inequalities $3\le t(\IR^2)\le c_2^B(\IR^2)\le c_2(\IZ^2)\le 3$ follow from  Theorem~\ref{sandwich}, \ref{t5}(2) and Proposition~\ref{p1}.

We claim that $c_2^B(\IR^\w)\ge 3$. Assuming that $c_2^B(\IR^\w)<3$ we conclude that $rc_k^B(\IR^\w)\le c_2^B(\IR^\w)-1\le 1$. Then by the Stabilization Lemma~\ref{l1:stab}, we get that $c_2(\IR^1)=c_2(\IR^\w)$ is finite. On the other hand, the real line has the 2-coloring $\chi:\IR\to 2$, $\chi^{-1}(1)=(0,\infty)$,  without unbounded monochromatic symmetric subsets. This coloring witnesses that $c_2(\IR^1)=\binfty$ and this is a contradiction. Therefore, 
$$3\le c_2^B(\IR^\w)\le c_2^B(\IR^n\times\IZ^{m-n})\le c_2^B(\IR^n\times\IZ^{m-n})\le c_2(\IZ^2)=3.$$
\smallskip

3. Assume that $n+m\ge 3$. Lemma~\ref{l:c3=6} and Theorem~\ref{sandwich} imply the inequalities 
$$6\le c_3^B(\IR^m)\le c_3^B(\IR^n\times\IZ^{m-n})\le c_3^B(\IR^n\times\IZ^{m-n})\le c_3(\IZ^3)=6$$
that turn into equalities.
\smallskip

4. Assume that $n+m=4$. Theorem~\ref{sandwich}, \ref{t5}(4) and Proposition~\ref{p1} imply the inequalities $$12\le t(\IR^4)\le c_4^B(\IR^4)\le  c_4^B(\IR^n\times\IZ^m)\le c_4(\IR^n\times\IZ^m)\le c_4(\IZ^4)\le 12,$$which actually are equalities. 
\smallskip

5. We need to prove that $c_k^B(\IR^n\times\IZ^m)=\binfty$ if $k\ge n+m+1<\w$.
This equality will follow as soon as we check that $c_k^B(\IR^{n+m})=\binfty$. 
Let $\Delta$ be a simplex in $\IR^{n+m}$ centered at the origin. Write the boundary $\partial \Delta$ as the union $\partial \Delta=\bigcup_{i=0}^{n+m}\Delta_i$ of its facets. Define a Borel $k$-coloring $\chi:\IR^n\to\{0,\dots,n+m\}\subset k$ assigning to each point $x\in\IR^n\setminus\{0\}$ the smallest number $i\le n+m$ such that the ray $\IR_+\cdot x$ meets the facet $\Delta_i$. Also put $\chi(0)=0$. It is easy to check that the coloring $\chi$ witnesses that the set $\IR^{n+m}$ is not $k$-centerpole for Borel colorings of $\IR^{n+m}$ and consequently, $c_k^B(\IR^{n+m})=\binfty$.
\smallskip

6. Assuming that $k\ge n+m+1$, we shall show that $c_k(\IR^n\times\IZ^m)=\binfty$. If $n+m$ is finite, then this follows from the preceding item. So, we assume that $n+m$ is infinite. Then the group $G=\IR^n\times\IZ^m$ has cardinality $2^{n+m}$. By Theorem 4 of \cite{BP}, for the group $G$ endowed with the discrete topology, we get $\nu(G)=\log |G|=\min\{\gamma:2^\gamma\ge |G|\}\le n+m\le k$, which means that $G$ admits a $k$-coloring without infinite monochromatic symmetric subset. This implies that the set $G$ is not $k$-centerpole in $G$ and thus $c_k(G)=\binfty$. 
\smallskip

7. Assume that $n+m\ge \w$ and $\w\le k<\cov(\M)$. The lower bound from Theorem~\ref{t:bounds}(3) implies that $\w\le c_k^B(\IR^\w)\le c_k^B(\IZ^\w)$. The upper bound $c_\kappa^B(\IZ^\w)\le\w$ will follow as soon as we check that each countable dense subset $C\subset\IZ^\w$ is $\kappa$-centerpole for Borel colorings of $\IZ^\w$. Let
$\chi:\IZ^\w\to\kappa$ be a Borel $\kappa$-coloring of $\IZ^\w$. Taking into account that $\IZ^\w=\bigcup_{i\in\kappa}\chi^{-1}(i)$ is homeomorphic to a dense $G_\delta$-subset of the real line, we conclude that for some color $i\in\kappa$ the preimage $A=\chi^{-1}(i)$ is not meager in $\IZ^\w$. Being a Borel subset of $\IZ^\w$, the set $A$ has the Baire property, which means that for some open subset $U\subset\IZ^\w$ the symmetric difference $A\triangle U$ is meager in $\IZ^\w$. Since $A$ is not meager, the set $U$ is not empty. Take any point $c\in U\cap C$ and observe that $V=U\cap (2c-U)$ is an open symmetric neighborhood of $c$. It 
follows that for the set $B=A\cap (2c-A)$ the symmetric difference $B\triangle V$ is meager. Since $V$ is not meager in $\IZ^\w$, the set $B$ is not meager and hence is unbounded in $\IZ^\w$ (since totally bounded subsets of $\IZ^\w$ are nowhere dense in $\IZ^\w$). Now we see that $B=A\cap (2c-A)$ is a monochromatic unbounded subset, symmetric with respect to the point $c$, witnessing that  the set $C$ is $\w$-centerpole for Borel coloring of $\IZ^\w$.

\section{Proof of Theorem~\ref{t:ckG}}\label{s:Pf:lc}

Let $k\ge 2$ be a finite cardinal  number and $G$ be an abelian $\LLC$-group with totally bounded Boolean subgroup $G[2]$ and ranks $n=r_\IR(G)$ and $m=r_\IZ(G)$. Let $\bar G$ be the completion of the group with respect to its (two-sided) uniformity.

First we give a proof the statements (3) and (4) of Theorem~\ref{t:ckG} holding under the additional assumption of the metrizability of the group $G$.

Since $c_k^B(\IR^n\times\IZ^{m-n})<\w$ iff $k\le m$, the Borel version of Theorem~\ref{t:ckG} will follow as soon as we prove two inequalities:

1) $c_k^B(G)\le c_k^B(\IR^n\times\IZ^{m-n})$ if $k\le m$ and
 
2) $c_k^B(\IR^n\times\IZ^{m-n})\le c_k^B(G)$ if $c_k^B(G)$ is finite.
\smallskip

1. Assume that $k\le m$. If the $\IZ$-rank $m=r_\IZ(G)$ is finite, then so is the $\IR$-rank $n=r_\IR(G)$ and we can find copies of the topological groups $\IR^n$ and $\IZ^m$ in $G$. Now consider the closure $H$ of the subgroup $\IR^n+\IZ^m$ in $G$. Since $G$ is an $\LLC$-group and $\IR^n+\IZ^m$ contains a dense finitely generated subgroup, the group $H$ is locally compact.
By the structure theorem of locally compact abelian groups \cite[Theorem 25]{Morris}, $H$ is topologically isomorphic to $\IR^r\oplus Z$ for some $r\in\w$ and a closed subgroup $Z\subset H$ that contains an open compact subgroup $K$. It follows from the inclusion $\IR^n\subset H$ that $n\le r$. On the other hand, $r\le r_\IZ(G)=n$. By the same reason, $r_\IZ(H)=m=r_\IZ(G)$. In particular, $r_\IZ(Z)=m-n$ and hence $H$ contains an isomorphic copy of the group $\IR^n\times\IZ^{m-n}$.  
Now we see that $r_k^B(G)\le r_k^B(\IR^n\times\IZ^{m-n})$.
\smallskip

Next, assume that the $\IZ$-rank $m=r_\IZ(G)$ is infinite but $n=r_\IR(G)$ is finite. By the Stabilization Lemma~\ref{l1:stab}, $c_k^B(\IR^n\times\IZ^{m-n})=c_k^B(\IR^n\times\IZ^\w)=c_k^B(\IR^n\times\IZ^{r-n})$ for $r=rc^B_k(\IR^n\times\IZ^\w)\le c_k^B(\IR^n\times\IZ^\w)<\binfty$. Repeating the above argument we can find a copy of the group $\IR^n\oplus \IZ^{s-n}$ in $G$ for some finite $s\ge r$ and conclude that $c_k^B(G)\le c_k^B(\IR^n\times\IZ^{s-n})\le c_k^B(\IR^n\times\IZ^{r-s})=c_k^B(\IR^n\times\IZ^{m-n})$. 
\smallskip

Finally, assume that the $\IR$-rank $n=r_\IR(G)$ is infinite. Then $c_k^B(\IR^n\times\IZ^{m-n})=c_k^B(\IR^\w)=c_k^B(\IR^r)$ for $r=rc_k^B(\IR^\w)\le c_k^B(\IR^\w)<\w$. By the definition of the $\IR$-rank $r_\IR(G)=n=\w$, we can find a copy of the group $\IR^r$ in $G$ and conclude that $c_k^B(G)\le c_k^B(\IR^r)=c_k^B(\IR^n\times\IZ^{m-n})$. This completes the proof of the inequality $c_k^B(G)\le c_k^B(\IR^n\times\IZ^{m-n})$.

Deleting the adjective ``Borel'' from the above proof we get the proof of the inequality $c_k(G)\le c_k(\IR^n\times\IZ^{m-n})$ holding for each $k\le m$.
\smallskip

2. Now assuming that $c_k^B(G)$ is finite and the group $G$ is metrizable, we prove the inequality $c_k^B(\IR^n\times\IZ^{m-n})\le c_k^B(G)$.

Fix a $k$-centerpole subset $C\subset G$ for Borel colorings of $G$ with cardinality $|C|=c_k^B(G)$. The subgroup $G[2]$ is totally bounded and hence has compact closure $K_2$ in the completion $\bar G$ of the group $G$. It follows that $K_2\subset\bar G[2]$. Since $G$ is an $\LLC$-group, the finitely-generated subgroup $\langle C\rangle$ has locally compact closure $\overline{\langle C\rangle}$ in $G$. It follows from the compactness of the subgroup $K_2$ that the sum $H=\overline{\langle C\rangle}+K_2$ is a locally compact subgroup of $\bar G$. This subgroup is compactly generated because it contains a dense subgroup generated by the compact set $C+K_2$. 

By the Structure Theorem for compactly generated locally compact abelian groups \cite[Theorem 24]{Morris}, $H$ is topologically isomorphic to $\IR^r\oplus \IZ^{s-r}\oplus K$ for some compact subgroup $K$ that contains all torsion elements of $H$. In particular, $K_2\subset K$. Now we see that the subgroup $2H=\{2x:x\in H\}$ is closed in $H$ and consequently, the subgroup $2H\cap G$ is closed in $G$. 
The group $G$ is metrizable and so is the quotient group $G/2H$. Then the subspace $X=(G/2H)\setminus(H/2H)$ is metrizable and thus paracompact. Since $H\supset G[2]$ we can apply Lemma~\ref{l:subB} and conclude that the set $C$ is $k$-centerpole for Borel colorings of the subgroup $H\cap G$. Since $H\cap G\subset H$, the set $C$ is $k$-centerpole for Borel colorings of the group $H$.

The compactness of the subgroup $K\subset H$ implies that the image $q(C)$ of $C$ under the quotient map $q:H\to H/K$ is a $k$-centerpole set for Borel colorings of the quotient group $H/K=\IR^r\times\IZ^{s-r}$. Since $H=\overline{\langle C\rangle}+K_2$ and $K_2\subset K$, we conclude that $\overline{\langle C\rangle}/(\overline{\langle C\rangle}\cap K)=q(\overline{\langle C\rangle})=H/K=\IR^r\times\IZ^{s-r}$ and hence $r\le n$ and $s\le m$. Consequently, $\IR^r\times\IZ^{s-r}\hookrightarrow \IR^n\times\IZ^{m-n}$ and $$c_k^B(\IR^n\times\IZ^{m-n})\subset c_k^B(\IR^r\times\IZ^{s-r})=c_k^B(H/K)\le|C|=c_k^B(G).$$

Deleting the adjective ``Borel'' from the above proof and applying Lemma~\ref{l:sub} instead of Lemma~\ref{l:subB}, we get the proof of the inequality
$c_k(\IR^n\times\IZ^{m-n})\le c_k(G)$ under the assumption that the number $c_k(G)$ is finite. Since Lemma~\ref{l:sub} does not require the metrizability of $G$, this upper bound hold without this assumption.

\section{Proof of Proposition~\ref{p:rc}}\label{s:p:rc}

Let $G$ be a metrizable abelian $\LLC$-group with totally bounded Boolean subgroup $G[2]$ and $k\in\IN$ be  such that $2\le k\le r_\IZ(G)$.
Theorems~\ref{t:ckG} and \ref{t:bounds} guarantee that $c_k^B(G)=c_k^B(\IR^n\times\IZ^{m-n})<\binfty$ where $n=r_\IZ(G)$ and $m=r_\IZ(G)$.

Let $r=rc_k(G)$ and $C\subset G$ be a subset of cardinality $|C|=c_k^B(G)$ such that $r_\IZ(\langle C\rangle)=r$. Without loss of generality, $0\in C$.
Since $G$ is an $\LLC$-group, the finitely generated subgroup $\langle C\rangle$ has locally compact closure in $G$. 

The totally bounded Boolean subgroup $G[2]$ has compact closure $K_2$ in the completion $\bar G$ of the abelian topological group $G$. It follows that the subgroup $H=\overline{\langle C\rangle}+K_2$ of $\bar G$ is locally compact and  compactly generated. Consequently, it contains a compact subgroup $K\supset K_2$ such that the quotient group $H/K$ is topologically isomorphic to $\IR^s\times\IZ^{r-s}$ for some $r\le s$. It follows from Lemma~\ref{l:sub} that the set $C$ is $k$-centerpole for Borel colorings of the group $H$. The compactness of the subgroup $K\subset H$ implies that the image $q(C)\subset H/K$ of $C$ 
under the quotient homomorphism $q:H\to H/K$ is a $k$-centerpole set for Borel colorings of $H/K$. Consequently, $$c_k^B(\IR^r)\le c_k^B(\IR^s\times\IZ^{r-s})=c_k^B(H/K)\le|q(C)|\le|C|=c_k^B(G)<\binfty$$ and hence $r\ge k$ by Theorem~\ref{t:bounds}(5). 

Now assume that $k\ge 4$. Since the set $q(C)$ is $k$-centerpole for Borel colorings of $H/K=\IR^s\times\IZ^{r-s}\subset\IR^r$, Lemma~\ref{l:afhull} implies that the affine hull
of $q(C)$ in the linear space $\IR^r$ has dimension $\le |q(C)|-3$. Since $0\in q(C)$, the affine hull of the set $q(C)$ coincides with its linear hull. Consequently, $r=r_\IZ(\langle C\rangle)=r_\IZ(\langle q(C)\rangle)\le |q(C)|-3\le|C|-3=c_k^B(G)-3$. 
This completes the proof of the lower and upper bounds
$$k\le rc_k(G)\le c_k^B(G)-3$$for all $k\ge 3$.

Next,we show that $rc_k(G)=k$ for $k\in\{2,3\}$. In this case $c_k^B(G)=c_k(\IZ^k)$ by Theorems~\ref{t:ckG} and \ref{t:exact}. Since $r_\IZ(G)\ge k$, the group $G$ contains an isomorphic copy of the group $\IZ^k$. Then each $k$-centerpole subset $C\subset \IZ^k\subset G$ with $|C|=c_k(\IZ^k)$ is $k$-centerpole for Borel colorings of $G$ and thus $k\le rc_k^B(G)\le r_\IZ(\langle C\rangle)\le k$, which implies the desired equality $rc_k^B(G)=k$.

\section{Proof of Stabilization Theorem~\ref{t:stab}}\label{s:t:stab}

Let $k\ge2$ and $G$ be an abelian $\LLC$-group with totally bounded Boolean subgroup $G[2]$. Let $n=r_\IR(G)$ and $m=r_\IZ(G)$. 
\smallskip

1. Assume that $m=r_\IZ(G)\ge rc_k^B(\IZ^\w)$. By Proposition~\ref{p:rc}, $k\le rc_k^B(\IZ^\w)\le r_\IZ(G)$ and then $c_k(G)=c_k(\IR^n\times\IZ^{m-n})$ by Theorem~\ref{t:ckG}.
Since $m=r_\IZ(\IR^n\times\IZ^{m-n})\ge rc_k^B(\IZ^\w)$, Lemma~\ref{l2:stab} guarantees that $c_k(G)=c_k^B(\IR^n\times\IZ^{m-n})=c_k^B(\IZ^\w)$.
\smallskip

2. Assume that the group $G$ is metrizable and $r_\IZ(G)\ge rc_k^B(\IR^n\times\IZ^\w)$. By Proposition~\ref{p:rc}, 
$k\le rc_k^B(\IR^n\times\IZ^\w)\le r_\IZ(G)= m$ and hence $c_k^B(G)=c_k^B(\IR^n\times\IZ^{m-n})$ by Theorem~\ref{t:ckG}.
Since $m=r_\IZ(\IR^n\times\IZ^{m-n})\ge rc_k^B(\IR^n\times\IZ^\w)$, Lemma~\ref{l1:stab} guarantees that $c_k^B(G)=c_{\IZ}^B(\IR^n\times\IZ^{m-n})=c_k^B(\IR^n\times\IZ^\w)$.
\smallskip

3. By analogy with the preceding case we can prove that 
$c_k^B(G)=c_k^B(\IR^\w)$ if $G$ is metrizable and $r_\IR(G)\ge rc^B_k(\IR^\w)$.

\section{Selected Open Problems}\label{s:OP}

By Theorem~\ref{t:exact}, $c_k^B(\IR^\w)=c_k(\IZ^\w)=c_k(\IZ^k)$ for all $k\le 4$.

\begin{problem} Is $c_k(\IZ^\w)=c_k(\IZ^k)$ for all $k\in\IN$? In particular, is $c_4(\IZ^n)=12$ for every $n\ge 4$?
\end{problem}

\begin{problem} Is $c_k^B(\IR^n)=c_k(\IR^n)$ for every $k\le n$?
\end{problem}

Theorem~\ref{t:bounds} gives an upper and lower bounds for the numbers $c_k(\IZ^k)$ that have exponential and polynomial growths, respectively.

\begin{problem} Is the growth of the sequence $\big(c_n(\IZ^n)\big)_{n\in\IN}$ exponential? 
\end{problem}

By \cite{Ba}, for every $k\in\{1,2,3\}$ any $k$-centerpole subset $C\subset\IZ^k$ of cardinality $|C|=c_k(\IZ^k)$ is affinely equivalent to the $\binom{k-1}{k-3}$-sandwich $\Xiup^{k-1}_{k-3}$.

\begin{problem} Is each 12-element 4-centerpole subset of $\IZ^4$ affinely equivalent to the $\binom{3}{1}$-sandwich $\Xiup^3_1$?
\end{problem}

It follows from the proof of Theorem 1 in \cite{Gr} that the free group $F_2$ with two generators and discrete topology has $c_2(F_2)\le 13$.

\begin{problem} What is the value of the cardinal $c_2(F_2)$? Is $c_3(F_2)$ finite?
\end{problem}

The last problem can be posed in a more general context.

\begin{problem} Investigate the cardinal characteristics $c_k(G)$ and $c_k^B(G)$ for non-commutative topological groups $G$.
\end{problem}

\end{document}